\documentclass[11pt,reqno]{amsart}

\usepackage{lineno,hyperref}

\modulolinenumbers[5]

\usepackage{amsmath}
\usepackage{amssymb}
\usepackage{mathrsfs}
\usepackage{amsthm}
\usepackage{bm}
\usepackage{graphicx}
\usepackage{booktabs}
\usepackage{float}
\usepackage{subfig}
\usepackage{geometry}
\usepackage{color}

\newtheorem{Def}{Definition}[section]
\newtheorem{lem}[Def]{Lemma}
\newtheorem{theo}[Def]{Theorem}
\newtheorem{pro}[Def]{Proposition}
\newtheorem{rem}[Def]{Remark}
\newtheorem{ex}[Def]{Example}
\newtheorem{assum}{Assumption}
\newtheorem{cor}[Def]{Corollary}

\newcommand{\xde}{Schr\"odinger~}
\newcommand{\tr}{\text{tr}}
\newcommand{\mcal}{\mathcal}

\newcommand{\mbb}{\mathbb}
\newcommand{\mbf}{\mathbf}

\newcommand{\ud}{\mathrm d}

\numberwithin{equation}{section}
\allowdisplaybreaks[4]

\begin{document}

\title[LDP]{Large deviations principles for symplectic discretizations of stochastic linear Schr\"odinger Equation}

\author{Chuchu Chen}
\address{Academy of Mathematics and Systems Science, Chinese Academy of Sciences, Beijing
	100190, China; School of Mathematical Sciences, University of Chinese Academy of
	Sciences, Beijing 100049, China}
\email{chenchuchu@lsec.cc.ac.cn}
\author{Jialin Hong}
\address{Academy of Mathematics and Systems Science, Chinese Academy of Sciences, Beijing
	100190, China; School of Mathematical Sciences, University of Chinese Academy of
	Sciences, Beijing 100049, China}
\email{hjl@lsec.cc.ac.cn}

\author{Diancong Jin}
\address{Academy of Mathematics and Systems Science, Chinese Academy of Sciences, Beijing
	100190, China; School of Mathematical Sciences, University of Chinese Academy of
	Sciences, Beijing 100049, China}
\email{diancongjin@lsec.cc.ac.cn (Corresponding author)}

\author{Liying Sun}
\address{Academy of Mathematics and Systems Science, Chinese Academy of Sciences, Beijing
	100190, China; School of Mathematical Sciences, University of Chinese Academy of
	Sciences, Beijing 100049, China}
\email{liyingsun@lsec.cc.ac.cn}

\thanks{This work is supported by National Natural Science Foundation of China (Nos. 91630312, 11971470, 11871068,  11926417, 11711530017).}


\begin{abstract}
In this paper, we consider the large deviations principles (LDPs) for the stochastic linear Schr\"odinger equation and its symplectic discretizations.
These numerical discretizations are the spatial semi-discretization based on spectral Galerkin method, and the further   full discretizations with symplectic schemes in temporal direction.
First, by means of the abstract G\"artner--Ellis theorem, we prove that the observable $B_T=\frac{u(T)}{T}$, $T>0$ of the exact solution $u$ is exponentially tight and satisfies an LDP on $L^2(0, \pi; \mathbb C)$.
Then, we present the LDPs for both
$\{B^M_T\}_{T>0}$ of the spatial discretization $\{u^M\}_{M\in\mbb N}$  and $\{B^M_N\}_{N\in \mbb N}$ of the full discretization $\{u^M_N\}_{M,N\in\mbb N}$, where $B^M_T=\frac{u^M(T)}{T}$ and $B^M_N=\frac{u^M_N}{N\tau}$ are the discrete approximations of $B_T$. 
Further, we show that both the semi-discretization $\{u^M\}_{M\in \mathbb N}$ and the full discretization $\{u^M_N\}_{M,N\in \mathbb N}$  based on temporal symplectic schemes  can weakly asymptotically preserve the LDP of $\{B_T\}_{T>0}$. 
These results show the ability of symplectic discretizations to preserve the LDP of the stochastic linear \xde equation, and first provide an effective approach to approximating the LDP rate function in infinite dimensional space based on the numerical discretizations.

\end{abstract}

\maketitle

\section{Introduction} \label{Sec1}

The stochastic Schr\"odinger equation, as an important stochastic Hamiltonian partial differential equation, is widely used to model the propagation of dispersive waves in inhomogeneous or random media (see e.g., \cite{Xde94}), and possesses the infinite dimensional stochastic symplectic geometric structure. 
To numerically inherit the geometric structure of the stochastic Schr\"odinger equation, \cite{ChenH} proposes the infinite dimensional stochastic symplectic algorithms and considers the semi-discretizations, such as the stochastic symplectic Runge--Kutta methods.
Moreover, the full discretizations based on the stochastic symplectic methods in temporal direction are also proposed  (see e.g., \cite{ChenH,CuiHLZ17,CuiHLZ19,HongW,JiangWH} and  references therein).
The numerical experiments show that the stochastic symplectic discretizations are more stable in the long-time simulation than the non-symplectic ones. 
In this paper, we aim to deepen  the understanding of the long-time asymptotical behavior and  probabilistic characteristics of stochastic symplectic methods from the perspective of LDP. More precisely, we study the LDPs for both the stochastic linear Schr\"odinger equation and its numerical discretizations, and  investigate the ability of symplectic discretizations to asymptotically preserve the LDP of the original system.

	
The theory of large deviations  has been applied to many other branches of sciences, for example statistical physics, finance, engineering information theory  (\cite{LDPapp09,LDP08}). It is concerned with the exponential decay of probabilities of very rare events, where the decay rate  is characterized by the LDP rate function. In some cases, LDP rate functions  describe steady rate and fluctuations of physical quantities, such as the entropy or free energy of statistical systems (see e.g., \cite{LDPapp18}). 

In this paper, we consider the following stochastic linear \xde equation
\begin{align} \label{xde1}
\ud u&=\bm{i}\Delta u\ud t+\bm{i}\alpha \ud W(t),\qquad t>0,\\
u(0)&=u_0\in H^1_0(0,\pi),\nonumber
\end{align}
where $\alpha>0$, $\Delta$ is the Laplace operator with the Dirichlet boundary condition, and $W$ is an $L^2(0,\pi;\mathbb{R})$-valued $Q$-Wiener process defined on a complete filtered probability space $\left(\Omega,\mathscr{F},\{\mathscr{F}_t\}_{t\geq0},\mbf P\right)$ with $\{\mathscr{F}_t\}_{t\geq0}$ satisfying the usual conditions.
The mass $\|u\|_{H^0}^2=\int_{0}^{\pi}|u(x)|^2\,\ud x$ of \eqref{xde1} is an  important physical quantity with $H^0:=L^2(0, \pi; \mathbb C)$, which is conservative if $\alpha=0$. However, in the stochastic setting, it grows linearly in the mean sense, i.e., $\mbf E\|u(T)\|^2_{H^0}=\mbf E\|u_0\|^2_{H^0}+\alpha^2T\tr(Q)$. Markov's inequality yields that the quantity $\|B_T\|_{H^0}$ tends to zero in probability, where $B_T:=\frac{u(T)}{T}$. In order to characterize the speed of convergence or give an exponential tail estimate, we investigate the LDP of $\{B_T\}_{T>0}$ on $H^0$.
Our idea is to use the abstract G\"artner--Ellis theorem, which involves the existence of the logarithmic moment generating function and exponential tightness. 
The Gaussian property of the exact solution on $H^0$ with the real inner product is analyzed to give the logarithmic moment generating function of $\{B_T\}_{T>0}$. A prerequisite of the exponential tightness is to find the compact subset of $H^0$, under the non-compactness of the \xde group, such that the probabilities  of  $\{B_T\}_{T>0}$ escaping from the compact subset is exponentially small. This 
relies on two skills: One is that the regularity of $u$ on $H^1$ gives a series of compact sets in $H^0$, and the other is that the Fernique theorem yields the estimate of probability that $B_T$ hits these compact sets on an exponential  scale.
Utilizing the property of reproducing kernel Hilbert space, we obtain the explicit expression of the large deviations rate function $I$ of $\{B_T\}_{T>0}$.

The large deviations rate functions characterize the essential decay rate of the probability of rare events. It is important for a  numerical discretization to preserve the rate function in certain sense.
Thus, for a numerical discretization of \eqref{xde1}, it is natural to ask:
\begin{itemize}
\item[(P1)] Does the discrete approximation of $\{B_T\}_{T>0}$, associated with the numerical discretization of \eqref{xde1}, satisfy the LDP?
\item[(P2)] If so, which kind of numerical discretizations can  preserve the LDP of the original system, namely preserve the LDP rate function, exactly or asymptotically?
\end{itemize}  
This paper aims to deal with the above problems.
We are faced with two major difficulties in the numerical analysis. One is  how to define the preservation for the LDP of an infinite dimensional stochastic  differential equation  by its  numerical discretizations. Unlike the LDP of the original system in infinite dimensional spaces,  the space concerning the LDP of a numerical discretization is finite dimensional. Therefore one needs a  reasonable definition  to link these two spaces.
Another difficulty arises from the symplectic discretizations of the stochastic  \xde equation, including the general formulation in high dimensional case and the combination with the theory of large deviations. 

Concerning these issues,  we  first apply the  spectral Galerkin method to \eqref{xde1} and get the spatial semi-discretization (see \eqref{Galerkin}) \begin{align}\label{Galerkin1}
\ud u^M(t)&=\bm{i}\Delta_M u^M(t)\ud t+\bm{i}\alpha P_M\ud W(t),\qquad t>0, \\
u^M(0)&=P_Mu_0\in H_M.\nonumber
\end{align}
Here  $H_M=\text{span}\left\{e_1,e_2,\ldots,e_M\right\}$, where $e_k, k=1,2,\ldots$ are  the eigenfunctions of $Q$ and form an orthonormal basis of $H^0$. In fact,  \eqref{Galerkin1} is a symplectic discretization and can be rewritten into a stochastic Hamiltonian system (see \eqref{SHS}):
\begin{align}\label{SHS1}
\ud P^M(t)&=\mcal MQ^M(t)\ud t,\nonumber\\
\ud Q^M(t)&=-\mcal MP^M(t)\ud t+\alpha\mcal Q\ud \beta(t),
\end{align}
where $u^M=P^M+\bm{i}Q^M$.
We define by $B^M_T=\frac{u^M(T)}{T}$, $T>0$ a discrete approximation of the observable $B_T$ for \eqref{Galerkin1}.
Following the arguments of dealing with the LDP for $\{B_T\}_{T>0}$, we prove that for each $M\in\mbb N$, $\{B^M_T\}_{T>0}$ obeys an LDP on $H_M$ with the good rate function $\widetilde{I}^M$.
Note that $\widetilde{I}^M$ and $I$ have different domains, which brings the difficulty to define  and study the preservation of the LDP for  $\{B_T\}_{T>0}$ by $\{u^M\}_{M\in\mbb N}$. 
A possibility is to transfer the LDP of $\{B^M_T\}_{T>0}$  on $H_M$ to $H^0$. This can be solved by means of Lemma \ref{subLDP}  which reveals the relationship between LDPs of a stochastic process on some space  and that on subspaces. This is to say,  $\{B^M_T\}_{T>0}$ also satisfies the LDP on $H^0$ with a rate function $I^M$. However, we also note that the valid domain, on which $I^M$ takes finite values, is a proper subset of the valid domain of $I$. Hence, we introduce the definition of \emph{weakly asymptotical preservation for LDP} (see Definition \ref{Def4.2}) in the sense that $I$ is well approximated by $I^M$ for some sufficiently large $M$. Further, we prove that $\{u^M\}_{M\in\mbb N}$  weakly asymptotically preserves the LDP of $\{B_T\}_{T>0}$ based on the strong continuity of $\{P_M\}_{M\in\mbb N}$.

Next, we attempt to show that the full discretization based on a large class of temporal symplectic discretization can weakly asymptotically preserve the LDP of $\{B_T\}_{T>0}$. In order to give the general formula of symplectic discretizations for the high dimensional system \eqref{SHS1}, an argument of dimensionality reduction is applied. More precisely, we  divide \eqref{SHS1} into  $M$ subsystems (see \eqref{subSHS}). Then we obtain a class of full discretizations $\{u^M_n\}_{M,n\in\mbb N}$
based on the temporal symplectic discretizations of \eqref{SHS1} by combining the symplectic discretizations in \cite{LDPosc} for  every $2$-dimensional subsystem.  For this full discretization,  we define a discrete approximation $B^M_N=\frac{u^M_N}{N\tau}$ of $B_T$, with $\tau$ being the temporal stepsize, and give the LDP of $\{B^M_N\}_{N\in\mbb N}$ based on the G\"artner--Ellis theorem and the contraction theorem. 
Further, we study whether $\{u^M_n\}_{M,n\in\mbb N}$ can weakly asymptotically preserve the LDP (see Definition \ref{sec5def5.4}) of $\{B_T\}_{T>0}$, which depends on the asymptotical behavior of the modified rate function $I_{mod}^{M,\tau}$ of $\{B^M_N\}_{N\in\mbb N}$. Notice that $I^M$ is a good approximation of $I$, it suffices to prove that for each $M\in\mbb N$, $\{u^M_n\}_{n\in\mbb N}$ can asymptotically preserve the LDP of $\{B^M_T\}_{T>0}$, i.e., the modified rate function $I_{mod}^{M,\tau}$ converges to $I^M$ pointwise as $\tau$ tends to zero.
Similar to \cite{LDPosc}, 
under certain convergence condition of numerical approximations,  we obtain $\lim_{\tau\to 0}I_{mod}^{M,\tau}(\cdot)=I^M(\cdot)$.  Combining  the asymptotical convergence of $I^M$ to $I$, we deduce our main conclusion that the full discretization $\{u^M_n\}_{M,n\in\mbb N}$, based on the  the spatial spectral Galerkin approximation and temporal symplectic discretizations, can weakly asymptotically preserve the LDP of $\{B_T\}_{T>0}$. That is to say, we  obtain a good approximation of the LDP rate function  of $\{B_T\}_{T>0}$  based on the symplectic discretizations.  To the best of our knowledge, this is the first result of approximating the LDP rate function in infinite dimensional space based on the numerical discretizations. We partially answer the open problem proposed by \cite{LDPosc}.


The paper is organized as follows. In Section \ref{Sec2}, some useful notations and preliminaries are introduced. 
In Section \ref{Sec3}, we give an introduction on the LDP  in general topological vector spaces, and  prove that $\{B_T\}_{T>0}$ satisfies an LDP on $H^0$. The weakly asymptotical preservations of LDP for $\{B_T\}_{T>0}$ by the spectral Galerkin approximation and the further full discretizations based on the temporal symplectic discretizations are given in Sections \ref{Sec4} and \ref{Sec5}, respectively.  Section \ref{Sec6} generalizes the LDP of $\{B_T\}_{T>0}$ to the case of complex-valued noises.  Future work is discussed  in Section \ref{Sec7}.

\section{Preliminaries}  \label{Sec2}
We begin with  some notations. Throughout this paper, denote by $H^s=H^s(0,\pi)$ and $H^s(0,\pi;\mathbb{R})$, the classical Sobolev space of complex-valued functions and  the classical Sobolev space of real-valued functions, respectively. In particular, denote $H^0=L^2(0,\pi;\mathbb{C})$, $H^1_0(0,\pi)=\{f\in H^1(0,\pi)\left|f(0)=f(\pi)=0\right.\}$, $U^0=L^2(0,\pi;\mathbb{R})$ and $U^1=H^1(0,\pi;\mathbb{R})$. For a linear operator $A$ from some Hilbert space onto itself, let $\lambda_k(A)$ be the $k$th eigenvalue of  $A$. For a complex number $z$, let $\Re z$ and  $\Im z$ be its real part and imaginary part, respectively. And denote by $\bm{i}$ the imaginary unit. Let $\left(U,\|\cdot\|_U,\langle\cdot,\cdot\rangle_U\right)$ and $\left(H,\|\cdot\|_H,\langle\cdot,\cdot\rangle_H\right)$ be two separable Hilbert spaces. Then $\mathcal L_2(U,H)$ denotes the Banach spaces consisting of all the Hilbert--Schmidt operators from  $U$ to $H$, with the norm $\left\|A\right\|_{\mcal L_2(U,H)}=\left(\sum_{k=1}^{\infty}\left\|Af_k\right\|_H^2\right)^{\frac{1}{2}}$, where $\{f_k\}_{k\in\mbb N}$ is any  orthonormal basis of $U$. 
Denote the real inner product by $\left\langle f,g\right\rangle_\mathbb{R}=\Re\int_{0}^{\pi}f(x)\bar{g}(x)dx$, and the complex inner product by $\left\langle f,g\right\rangle_\mathbb{C}=\int_{0}^{\pi}f(x)\bar{g}(x)dx$ for $f$, $g\in H^0$.

For a given  $M\in\mbb N$, $\mbb C^M$ denotes the space of $M$-dimensional complex-valued vectors. Define the inner product on $\mbb C^M$ by $\langle u,v\rangle_{\mbb R}=\sum\limits_{k=1}^{M}\left(\Re u_k\Re v_k+\Im u_k\Im
v_k\right)$, and the norm by $\|u\|=\sqrt{\langle u,u\rangle_{\mbb R}}$ for any $u=\left(u_1,u_2,\ldots,u_M\right)$, $v=\left(v_1,v_2,\ldots,v_M\right)\in\mbb C^M$.  $R=\mcal{O}(h^p)$ stands for $\left|R\right|\leq Ch^p$, for all sufficiently small $h>0$. $f(h)\sim h^p$ means that $f(h)$ and $h^p$ are equivalent infinitesimal. For the random variables  $X,Y$, $\mbf{Var}(X)$ denotes the covariance operator of $X$ and $\mbf{Cor}(X,Y)$ denotes the correlation operator of $X$ and $Y$. 

In order to investigate the stochastic \xde equation \eqref{xde1}, we introduce the definition and properties of the noise.
Let $e_k(x)=\sqrt{\frac{2}{\pi}}\sin(kx)$, then $\{e_k\}_{k\in\mathbb{N}}$ forms an orthonormal basis of both $(H^0,\langle\cdot,\cdot\rangle_{\mbb C})$ and $(U^0,\langle\cdot,\cdot\rangle_{\mbb R})$. 
Assume that $Q$ is a nonnegative symmetric operator on $U^0$ with $Qe_k=\eta_k e_k$ for some  non-increasing sequence $\{\eta_k\}_{k\in\mbb N}$. Then $W$ has the  expansion $W(t)=\sum_{k\geq1}\sqrt{\eta_k}\beta_k(t)e_k$. $Q$ can  be extended to $H^0$ by defining $Qf=Q(\Re f)+\bm{i}Q(\Im f)$ for every $f\in H^0$ and the extended operator is still denoted by $Q$, if no confusion occurs. Noting that $\Delta e_k=-k^2e_k$, $k=1,2\ldots$,  we have that $\Delta Q=Q\Delta$.

Let $S(t)=e^{\bm{i}t\Delta}$ be the  unitary $C_0$-group generated by $A$. Throughout the paper, we assume that $Q^{\frac{1}{2}}\in \mathcal{L}_2(U^0,U^1)$ and $u_0\in H^1_0(0,\pi)$, then \eqref{xde1} admits a unique mild solution in $H^1_0(0,\pi)$ (see e.g., \cite{Cohen}):
\begin{align}\label{mild}
u(t)=S(t)u_0+\bm{i}\alpha\int_{0}^{t}S(t-s)\ud W(s).
\end{align} 
Next, we give some results about the property of the distribution of exact solution  \eqref{mild}. These results are based on the following proposition.
\begin{pro}\cite[Proposition 4.28]{Prato1}\label{Cor}
	Let $W$ be a $U$-valued $Q$-Wiener process and $\mcal N^2_{W}(0,T;L^2_0)$ denote the set
	\begin{align*}
	\left\{\left.\Phi:\left[0,T\right]\times \Omega\to\mcal L_2(Q^\frac{1}{2}(U),H)\right|\Phi ~\text{is predicable and} \mbf~ \mbf E\int_{0}^{T}\left\|\Phi(s)\circ Q^{\frac{1}{2}}\right\|^2_{\mcal L_2(U,H)}d s<\infty\right\},
	\end{align*}
	where $H$ is a separable Hilbert space.	
	Assume that $\Phi_1,\Phi_2\in\mcal N^2_W(0,T;L^2_0)$, then the correlation operators
	\begin{align*}
	V(t,s)=\mbf{Cor}(\Phi_1\cdot W(t),\Phi_2\cdot W(s)),\qquad t,s\in[0,T]
	\end{align*}
	are given by the formula
	\begin{align*}
	V(t,s)=\mbf E\int_{0}^{t\wedge s}\Phi_2(r) Q(\Phi_1(r))^*dr.
	\end{align*}
	Here, the operator $V(t,s)$ is defined by
	$$\left\langle V(t,s)a,b\right\rangle_H=\mbf E\left\langle \Phi_1\cdot W(t),a\right\rangle_H\left\langle \Phi_2\cdot W(s),b\right\rangle_H,\qquad a,b\in H.$$
\end{pro}
It follows from  \eqref{mild} that
\begin{align*}
u(t)&=S(t)u_0+\bm{i}\alpha\int_{0}^{t}\left(\cos((t-s)\Delta))+\bm{i}\sin((t-s)\Delta)\right)d W(s)\nonumber\\
&=S(t)u_0-\alpha\int_{0}^{t}\sin((t-s)\Delta)d W(s)+\bm{i}\alpha\int_{0}^{t}\cos((t-s)\Delta)d W(s)\nonumber\\
&=:S(t)u_0-\alpha W_{\sin}(t)+\bm{i}\alpha W_{\cos}(t).
\end{align*}
Noting that $\left\langle f,g\right\rangle_\mathbb{R}=\left\langle \Re f,\Re g\right\rangle_\mathbb{R}+\left\langle \Im f,\Im g\right\rangle_\mathbb{R}$, we have that for each $h=\Re h+\bm{i}\Im h\in H^0$,
\begin{align}\label{ut,h}
\left\langle u(t),h\right\rangle_\mathbb{R}=\left\langle S(t)u_0,h\right\rangle_\mathbb{R}-\alpha\left\langle W_{\sin}(t),\Re h\right\rangle_\mathbb{R}+\alpha\left\langle W_{\cos}(t),\Im h\right\rangle_\mathbb{R}.
\end{align}
Hence, 
\begin{align}\label{Eut,h}
\mbf E\left\langle u(t),h\right\rangle_\mathbb{R}=\left\langle S(t)u_0,h\right\rangle_\mathbb{R}.
\end{align}
It follows from Proposition \ref{Cor} that
\begin{gather}
W_{\sin}(t)\sim\mcal N\left(0,\int_{0}^{t}\sin^2((t-s)\Delta)Qds\right),\qquad W_{\cos}(t)\sim\mcal N\left(0,\int_{0}^{t}\cos^2((t-s)\Delta)Qds\right), \label{Wsinlaw}\\
\mbf{Cor}\left(W_{\sin}(t),W_{\cos}(t)\right)=\int_0^t\sin((t-s)\Delta)\cos((t-s)\Delta)Qds. \nonumber
\end{gather}
Using the above formulas and $\Delta Q=Q\Delta$, one has                 
\begin{align}\label{varut,h}
\mbf{Var}\left\langle u(t),h\right\rangle_\mathbb{R}=&\alpha^2\left\langle \int_{0}^{t}\sin^2((t-s)\Delta)Qds\Re h,\Re h\right\rangle_\mathbb{R}+\alpha^2\left\langle \int_{0}^{t}\cos^2((t-s)\Delta)Qds\Im h,\Im h\right\rangle_\mathbb{R}\nonumber\\
&-2\alpha^2\left\langle \int_{0}^{t}\sin((t-s)\Delta)\cos((t-s)\Delta)Qds\Re h,\Im h\right\rangle_\mathbb{R}.
\end{align}
Since $\Delta$ is invertible, we have
\begin{gather}
\int_{0}^{t}\sin^2((t-s)\Delta)ds=\frac{1}{2}\int_{0}^{t}\left(I-\cos(2(t-s)\Delta)\right)ds=\frac{tI}{2}-\frac{\Delta^{-1}}{4}\sin(2t\Delta),\label{sinint}\\
\int_{0}^{t}\cos^2((t-s)\Delta)ds=\frac{1}{2}\int_{0}^{t}\left(I+\cos(2(t-s)\Delta)\right)ds=\frac{tI}{2}+\frac{\Delta^{-1}}{4}\sin(2t\Delta),\label{csoint}\\
\int_{0}^{t}\sin(2(t-s)\Delta)ds=\frac{\Delta^{-1}}{2}\left[I-\cos(2t\Delta)\right] \label{sincosint}.
\end{gather}
Combining \eqref{varut,h}, \eqref{sinint}, \eqref{csoint} and \eqref{sincosint} leads to
\begin{align}\label{varut,h1}
\mbf{Var}\left\langle u(t),h\right\rangle_\mathbb{R}=&\frac{t\alpha^2}{2}\left(\left\langle Q\Re h,\Re h\right\rangle_\mathbb{R}+\left\langle Q\Im h,\Im h\right\rangle_\mathbb{R}\right)-\frac{\alpha^2}{4}\left[\left\langle \Delta^{-1}\sin(2t\Delta)Q\Re h,\Re h\right\rangle_\mathbb{R}\right.\nonumber\\
&\left.-\left\langle \Delta^{-1}\sin(2t\Delta)Q\Im h,\Im h\right\rangle_\mathbb{R}\right]-\frac{\alpha^2}{2}\left\langle\Delta^{-1} \left(I-\cos(2t\Delta)\right)Q\Re h,\Im h\right\rangle_\mathbb{R}.
\end{align}

\section{LDP for $B_T$ of stochastic linear \xde equation }\label{Sec3}
In this section, we study the LDP for $\{B_T\}_{T>0}$ by means of the abstract G\"artner--Ellis theorem. As a corollary, we give the exponential tail estimate of the mass of \eqref{xde1}. Throughout this section, let $\mathcal{X}$ be a locally convex Hausdorff topological vector space and $\mathcal{X}^*$ be its dual space. 
\subsection{Introduction on LDP}
In this part, we recall some  concepts upon LDP and useful theorems and lemmas in studying the LDP of a family of probability measures.
First we introduce the definitions of rate function and LDP (see e.g., \cite{LDPosc}).
\begin{Def}\label{ratefun}
	A real-valued function $I:\mcal X\rightarrow[0,\infty]$ is called a rate function, if it is lower semicontinuous, i.e., for each $a\in[0,\infty)$, the level set $I^{-1}([-\infty,a])$ is a closed subset of $\mcal X$. If all level sets $I^{-1}([-\infty,a])$, $a\in[0,\infty)$, are compact, then $I$ is called a good rate function.
\end{Def}

\begin{Def}\label{LDPdef}
	Let $I$ be a rate function and $\{\mu_\epsilon\}_{\epsilon>0}$ be a family of probability measures on $\mcal X$. We say that $\{\mu_\epsilon\}_{\epsilon>0}$ satisfies an LDP on $\mcal X$ with the rate function $I$ if
	\begin{flalign}
	(\rm{LDP 1})\qquad \qquad&\liminf_{\epsilon\to 0}\epsilon\ln(\mu_\epsilon(U))\geq-\inf I(U)\qquad\text{for every open}~ U\subset \mcal X,\nonumber&\\
	(\rm{LDP 2})\qquad\qquad &\limsup_{\epsilon\to 0}\epsilon\ln(\mu_\epsilon(C))\leq-\inf I(C)\qquad\text{for every closed}~ C\subset \mcal X.&\nonumber
	\end{flalign}
\end{Def} 
Analogously, we say that a family of random variables  $\{Z_{\epsilon}\}_{\epsilon>0}$ valued on $\mcal X$ satisfies an LDP with the rate function $I$ if its distribution satisfies the lower bound LDP (LDP1) and upper bound LDP (LDP2) in Definition \ref{LDPdef} for the rate function $I$.

Generally speaking, we need to investigate the logarithmic moment generating function and the exponential tightness of $\{\mu_{\epsilon}\}_{\epsilon>0}$, when we derive the LDP of $\{\mu_{\epsilon}\}_{\epsilon>0}$. Especially, if the state space $\mcal X$ is finite dimensional,  the existence of logarithmic moment generating function implies the exponential tightness.  However, when $\mcal X$ is infinite dimensional, the exponential tightness of $\{\mu_{\epsilon}\}_{\epsilon>0}$ can not be ignored. 

\begin{Def}\cite[Page 8]{Dembo} \label{exptightdef}
	 A family of probability measures $\{\mu_{\epsilon}\}$ on $\mcal X$ is exponentially tight if for every $\alpha<\infty$, there exists a compact set $K_\alpha\subset\mcal X$ such that 
	\begin{align}\label{exptight}
		\limsup_{\epsilon\to 0}\epsilon\ln\mu_{\epsilon}(K_\alpha^c)<-\alpha.
	\end{align}
\end{Def}
\begin{theo}\cite[Corollary 4.6.14]{Dembo}\label{GE}
Let $\{\mu_\epsilon\}_{\epsilon>0}$ be an exponentially tight family of Borel probability measures on  $\mcal X$. Suppose the logarithmic moment generating function $\Lambda(\cdot)=\lim_{\epsilon\to 0}\epsilon\Lambda_{\mu_{\epsilon}}(\cdot/\epsilon)$ is finite valued and Gateaux differentiable, where $\Lambda_{\mu_{\epsilon}}(\lambda):=\ln\int_{\mcal X}e^{\lambda(x)}\mu_{\epsilon}(dx)$, $\lambda\in \mcal X^*$. Then  $\{\mu_\epsilon\}_{\epsilon>0}$ satisfies the LDP in $\mcal X$ with the convex, good rate function $\Lambda^*(x)=\underset{\lambda\in\mcal X^*}{\sup}\{\lambda(x)-\Lambda(\lambda)\}$.
\end{theo}
Theorem \ref{GE} can be viewed as the abstract G\"artner--Ellis theorem. The following two lemmas are  useful to derive new LDPs based on a given LDP. The first lemma is also called the contraction principle, which produces a new LDP on another space based on the known LDP via a continuous mapping. The second one  gives the relationship between the LDP of $\{\mu_{\epsilon}\}_{\epsilon>0}$ on $\mcal X$ and that on the subspaces of $\mcal X$.

\begin{lem}\cite[Theorem 4.2.1]{Dembo}\label{contraction}
	Let $\mcal Y$	be another Hausdorff topological space,  $f:\mcal X\to\mcal Y$ be a continuous function, and $I: \mcal X\to [0,\infty]$ be a good rate function.
	\begin{itemize}
		\item[(a)] For each $y\in\mcal Y$, define
		\begin{align*}
		\tilde{I}(y)\triangleq\inf\left\{I(x):~x\in\mcal X,\quad y=f(x)\right\}.
		\end{align*}
		Then $\tilde{I}(y)$ is a good rate function on $\mcal Y$, where as usual the infimum over the empty set is taken as $\infty$.
		\item[(b)] If $I$ controls the LDP associated with a family of probability measures $\{\mu_{\epsilon}\}$ on $\mcal X$, then $\tilde{I}(y)$ controls the LDP associated with  the family of probability measures $\left\{\mu_{\epsilon}\circ f^{-1}\right\}$ on $\mcal Y$.
	\end{itemize}
\end{lem}

\begin{lem}\cite[Lemma 4.1.5]{Dembo}\label{subLDP}
	Let $E$ be a measurable subset of  $\mcal X$ such that $\mu_{\epsilon}(E)=1$ for all $\epsilon>0$. Suppose that $E$ is equipped with the topology induced by $\mcal X$. If $E$ is a closed subset of $\mcal X$ and $\{\mu_\epsilon\}_{\epsilon>0}$ satisfies the LDP on $E$ with the rate function $I$, then $\{\mu_\epsilon\}_{\epsilon>0}$ satisfies the LDP on $\mcal X$ with the rate function $\tilde{I}(y)$ such that $\tilde{I}(y)=I$ on $E$ and $\tilde{I}(y)=\infty$ on $E^c$.
\end{lem}

\begin{pro} \cite[Lemma 1.2.15]{Dembo} \label{limsup}
	Let $N$ be a fixed integer. Then,  for every $a^i_{\epsilon}\geq0$,
	\begin{align*}
	\limsup_{\epsilon\to 0}\epsilon\ln\left(\sum_{i=1}^{N}a^i_{\epsilon}\right)=\underset{i=1,\ldots,N}{\max}\limsup_{\epsilon\to 0}\epsilon\ln a^i_{\epsilon}.
	\end{align*}
\end{pro}
Proposition \ref{limsup} is an important tool in deriving (LDP1) and (LDP2). Furthermore,  we need to make use of the following proposition in stochastic calculus. 

\begin{pro}\cite[Propostition 1.13]{Prato2}\label{Fernique}
	Assume that $\widetilde{Q}$ is a nonnegative symmetric operator on a real separable Hilbert space $H$ with finite trace. Let $\lambda_1\geq\lambda_2\geq\cdots\geq\lambda_n\geq\cdots$ be the eigenvalues of $\widetilde{Q}$. Define the determinant of $(I-2\epsilon \widetilde{Q})$ by setting $\det(I-2\epsilon \widetilde{Q}):=\lim_{n\to \infty}\prod_{k=1}^{n}(1-2\epsilon\lambda_k):=\prod_{k=1}^{\infty}(1-2\epsilon\lambda_k)$. Let $\mu=\mcal N(0,\widetilde{Q})$ be the symmetric Gaussian measure on $H$. Then for every $\epsilon\in\mbb R$, 
	\begin{align}
		\int_{H}e^{\epsilon\|x\|_H^2}\mu(dx)=
		\begin{cases}
		\left[\det(I-2\epsilon \widetilde{Q})\right]^{-1/2}, \qquad &\text{if}~\epsilon<\frac{1}{2\lambda_1},\\
		+\infty,&\text{otherwise}.
		\end{cases}
	\end{align}
\end{pro}

\subsection{LDP for $\{B_T\}_{T>0}$}
In this subsection, we show the LDP for $\{B_T\}_{T>0}$ of \eqref{xde1} by using Theorem \ref{GE}, where $B_T:=\frac{u(T)}{T}$ with $u(T)$ being the solution of \eqref{xde1} at time $T$. The regime of G\"artner--Ellis theorem is applicable to the real Banach space. Given that the exact solution $\{u(t)\}_{t\geq 0}$ takes values in $H^0$, the space of complex-valued functions, we use the real inner product  to establish the LDP of $\{B_T\}_{T>0}$ on $H^0$.
%
\begin{theo}\label{LDP for BT}
 	$\{B_T\}_{T>0}$ satisfies an LDP on $H^0$ with the good rate function
 	\begin{align*}
 	I(x)=\begin{cases}
 	\frac{1}{\alpha^2}\left\|Q^{-\frac{1}{2}}x\right\|_{H^0}^2, \qquad &\text{if}~x\in Q^{\frac{1}{2}}(H^0), \\
 	+\infty, &\text{otherwise},
 	\end{cases}
 	\end{align*}
where $Q^{-\frac{1}{2}}$ is the pseudo inverse of $Q^{\frac{1}{2}}$.
\end{theo}
\begin{proof}
	We divide the proof into three steps.\\
\emph{Step $1$: The logarithmic moment generating function of $\{B_T\}_{T>0}$}\\
For each $\lambda\in H^0$, define the mapping $\lambda^\prime: H^0\rightarrow \mbb R$ by $\lambda^\prime(x)=\left\langle x,\lambda\right\rangle_\mathbb{R}$, $x\in H^0$. Then by Riesz representation theorem, $\{\lambda'\}_{\lambda\in H^0}$ forms the set of  all real bounded linear functionals of $H^0$, i.e., $\{\lambda'\}_{\lambda\in H^0}=(H^0)^*$. 
Since $\langle u(t),\lambda\rangle_{\mbb R}$ is Gaussian, it follows from \eqref{Eut,h} and \eqref{varut,h1} that
\begin{align}
	\Lambda(\lambda')=&\lim_{T\to \infty}\frac{1}{T}\ln\mbf Ee^{T\langle B_T,\lambda\rangle_{\mbb R}}=\lim_{T\to \infty}\frac{1}{T}\ln\mbf Ee^{\langle u(T),\lambda\rangle_{\mbb R}} \nonumber\\
	=&\lim_{T\to \infty}\frac{1}{T}\left[\mbf E\langle u(T),\lambda\rangle_{\mbb R}+\frac{1}{2}\mbf{Var}\langle u(T),\lambda\rangle_{\mbb R}\right]\nonumber\\
	=&\frac{\alpha^2}{4}\left(\left\langle Q\Re \lambda,\Re \lambda\right\rangle_\mathbb{R}+\left\langle Q\Im \lambda,\Im \lambda\right\rangle_\mathbb{R}\right)\nonumber\\
	=&\frac{\alpha^2}{4}\left\|Q^{\frac{1}{2}}\lambda\right\|_{H^0}^2, \label{Lambda}
\end{align}
where we use the facts $\left\|\sin(t\Delta)\right\|_{\mcal L(H^0)}\leq1$, $\left\|\cos(t\Delta)\right\|_{\mcal L(H^0)}\leq1$ and $\left\|\Delta^{-1}\right\|_{\mcal L(H^0)}=1$.

\emph{Step $2$: Exponential tightness of $\{B_T\}_{T>0}$}\\
In order to obtain the exponential tightness of $\{B_T\}_{T>0}$ (see Definition \ref{exptightdef}), it suffices to show that there exists a family of compact sets $\{K_L\}_{L>0}$ such that 
\begin{align}\label{provetingt}
	\lim_{L\to \infty}\limsup_{T\to\infty}\frac{1}{T}\mbf \ln \mbf P\left(B_T\in K^c_L\right)=-\infty.
\end{align}
Define $K_L=\left\{f\in H^1\left|\right.\|f\|_{H^1}\leq L\right\}$. Then $K_L$ is the compact set of $H^0$. Recall that $u(T)=S(T)u_0-\alpha W_{\sin}(T)+\bm{i}\alpha W_{\cos}(T)$. Thus
\begin{align}
&\mbf P\left(B_T\in K_L^c\right)=\mbf P\left(\|u(T)\|_{H^1}>LT\right)\nonumber\\
\leq&\mbf P\left(\|S(T)u_0\|_{H^1}>\frac{TL}{3}\right)+\mbf P\left(\alpha\|W_{\sin}(T)\|_{U^1}>\frac{TL}{3}\right)+\mbf P\left(\alpha\|W_{\cos}(T)\|_{U^1}>\frac{TL}{3}\right). \label{716k1}
\end{align}
Since the first term in \eqref{716k1} is $0$ for sufficiently large $T$, we only need to estimate the second and third terms in \eqref{716k1}.

Since $Q^{\frac{1}{2}}\in\mcal L_2(U^0,U^1)$, $Q$ is also the finite trace operator on $U^1$. Then we obtain from \eqref{sinint} that
\begin{gather*}
	W_{\sin}(T)\sim\mcal N\left(0,\int_{0}^{t}\sin^2((t-s)\Delta)Qds\right)=\mcal N\left(0,\left(\frac{TI}{2}-\frac{\Delta^{-1}\sin(2T\Delta)}{4}\right)Q\right) \text{on $U^1$}.
\end{gather*}
Further, it holds that
\begin{align}
	\frac{W_{\sin(T)}}{\sqrt{T}}\sim\mcal N\left(0,\left(\frac{I}{2}-\frac{\Delta^{-1}\sin(2T\Delta)}{4T}\right)Q\right) \text{on $U^1$}.
\end{align}
By Markov's inequality, for each $\varepsilon>0$,
\begin{align}
	\mbf P\left(\alpha\|W_{\sin}(T)\|_{U^1}>\frac{TL}{3}\right)&=\mbf P\left(\left\|\frac{W_{\sin}(T)}{\sqrt{T}}\right\|_{U^1}>\frac{
		\sqrt{T}L}{3\alpha}\right)  \nonumber\\
	&=\mbf P\left(\exp\left\{\varepsilon\left\|\frac{W_{\sin}(T)}{\sqrt{T}}\right\|^2_{U^1}\right\}>\exp\left\{\frac{
	\varepsilon TL^2}{9\alpha^2}\right\}\right)\nonumber\\ 
	&\leq e^{-\frac{\varepsilon TL^2}{9\alpha^2}}\mbf E\exp\left\{{\varepsilon\left\|\frac{W_{\sin}(T)}{\sqrt{T}}\right\|_{U^1}^2}\right\}. \label{716k2}
\end{align}
Notice that $\lambda_k\left(\left(\frac{I}{2}-\frac{\Delta^{-1}\sin(2T\Delta)}{4T}\right)Q\right)=\left(\frac{1}{2}-\frac{\sin(2Tk^2)}{4Tk^2}\right)\eta_k=\frac{1}{2}\left(1-\frac{\sin(2Tk^2)}{2Tk^2}\right)\eta_k<\eta_k\leq\eta_1$. It follows from Proposition \ref{Fernique} that for each $0<\varepsilon<\frac{1}{2\eta_1}$, 
\begin{align}
	\mbf E\exp\left\{{\varepsilon\left\|\frac{W_{\sin}(T)}{\sqrt{T}}\right\|_{U^1}^2}\right\}&=\left[\det\left(I-2\varepsilon\left(\frac{I}{2}-\frac{\Delta^{-1}\sin(2T\Delta)}{4T}\right)Q\right)\right]^{-\frac{1}{2}}\nonumber\\
	&<\left[\det(I-2\varepsilon Q)\right]^{-\frac{1}{2}}=C(\varepsilon,Q),\label{716k3}
\end{align}
where we have used the fact that $\left[\det(I-2\varepsilon Q)\right]^{-\frac{1}{2}}=\left(\prod_{k=1}^{\infty}(1-2\varepsilon\eta_k)\right)^{-\frac{1}{2}}$ is monotonically increasing with respect to $\eta_k$ for every $k=1,2,\ldots$
Combining \eqref{716k3} with  \eqref{716k2} yields
\begin{align}
	\limsup_{T\to\infty}\frac{1}{T}\ln\mbf P\left(\alpha\|W_{\sin}(T)\|_{U^1}>\frac{TL}{3}\right)\leq\limsup_{T\to\infty}\frac{1}{T}\ln\left(e^{-\frac{\varepsilon TL^2}{9\alpha^2}}C(\varepsilon,Q)\right)={-\frac{\varepsilon L^2}{9\alpha^2}}. \label{917k1}
\end{align}

In addition, it holds that
\begin{align*}
\frac{W_{\cos(T)}}{\sqrt{T}}\sim\mcal N\left(0,\left(\frac{I}{2}+\frac{\Delta^{-1}\sin(2T\Delta)}{4T}\right)Q\right) \text{on $U^1$}.
\end{align*}
Then $\lambda_k\left(\left(\frac{I}{2}+\frac{\Delta^{-1}\sin(2T\Delta)}{4T}\right)Q\right)=\frac{1}{2}\left(1+\frac{\sin(2Tk^2)}{2Tk^2}\right)\eta_k<\eta_k\leq\eta_1$.
Analogous to the proof of \eqref{917k1}, one  has that for $0<\varepsilon<\frac{1}{2\eta_1}$, 
\begin{align}
\limsup_{T\to\infty}\frac{1}{T}\ln\mbf P\left(\alpha\|W_{\cos}(T)\|_{U^1}>\frac{TL}{3}\right)\leq{-\frac{\varepsilon L^2}{9\alpha^2}}. \label{917k2}
\end{align}
Combining \eqref{917k1}, \eqref{917k2}, \eqref{716k1} and Proposition \ref{limsup}, we obtain 
\begin{align*}
\limsup_{T\to\infty}\frac{1}{T}\ln\mbf P\left(B_T\in K_L^c\right)\leq \max\left\{-\frac{\varepsilon L^2}{9\alpha^2},-\frac{\varepsilon L^2}{9\alpha^2}\right\}=-\frac{\varepsilon L^2}{9\alpha^2}.
\end{align*}
Accordingly, we have
\begin{align}
	\lim_{L\to \infty}\limsup_{T\to\infty}\frac{1}{T}\ln\mbf P\left(B_T\in K_L^c\right)=-\infty,
\end{align}
which proves the exponential tightness of $\{B_T\}_{T>0}$.

It is verified that $\Lambda(\lambda')$ is finite valued and Gateaux differentiable. In fact, $\Lambda(\lambda')$ is Fr\'echet differentiable, and its Fr\'echet derivative is $\mcal D\Lambda(\lambda')(\cdot)=\frac{\alpha^2}{2}\langle Q\lambda,\cdot\rangle_{\mbb R}$.
Due to Theorem \ref{GE}, $\{B_T\}_{T>0}$ satisfies an LDP on $H^0$ with the good rate function $\Lambda^*$. It remains to give the explicit expression of the Fenchel--Legendre transform $\Lambda^*$ of $\Lambda$.

\emph{Step $3$: The explicit expression of  $\Lambda^*$}\\
Before giving the expression of  $\Lambda^*$, we recall the concept of reproducing kernel Hilbert space (RKHS).
Let $\mu$ be a centered Gaussian measure on a separable Banach space $E$. An arbitrary $\varphi\in E^*$  can be identified with an element of the Hilbert space $L^2(\mu):=L^2(E,\mcal B(E),\mu;\mbb R)$. Denote by $\overline{E^*}=\overline{E^*}^{L^2(\mu)}$ the closure of $E^*$ in $L^2(\mu)$. Define a mapping $J:\overline{E^*}\to E$ by setting
  \begin{align*}
  	J(\varphi)=\int_{E}x\varphi(x)\mu(dx),\qquad \forall\quad \varphi\in\overline{E^*}.
  \end{align*} 
  Then the image $\mcal H_{\mu}$ of $J$ in $E$, $\mcal H_{\mu}=J(\overline{E^*})$ is the RKHS of $\mu$ with the scalar product
  \begin{align*}
  	\left\langle J(\varphi),J(\psi)\right\rangle_{\mcal H_{\mu}}=\int_{E}\varphi(x)\psi(x)\mu(dx).
  \end{align*}
  Further, if $\mu=\mcal N\mcal(0,\widetilde Q)$ is a Gaussian measure on some Hilbert space $H$ with $\widetilde Q$ being a nonnegative symmetric  operator with finite trace, then the RKHS $\mcal H_{\mu}$ of $\mu$ 
  is $\mcal H_{\mu}=\widetilde Q^{\frac{1}{2}}(H)$ with the norm $\|x\|_{\mcal H_{\mu}}=\|\widetilde Q^{-\frac{1}{2}}x\|_H$.  We refer to \cite[Section 2.2.2]{Prato1} for more details of the RKHS. 

In our case,  $\mu=\mcal N(0,Q)$. The mapping $J: \overline{(H^0)^*}^{L^2(\mu)}\to H^0$ is
\begin{align*}
	J(h)=\int_{H^0}zh(z)\mu(dz).
\end{align*}
Then $\mcal H_{\mu}=J\left(\overline{(H^0)^*}^{L^2(\mu)}\right)=Q^{\frac{1}{2}}(H^0)$.
It follows from the properties of Gaussian measure that
\begin{align*}
	\int_{H^0}\left\langle\lambda,x\right\rangle_{\mbb R}^2\mu(dx)=\left\langle Q\lambda,\lambda\right\rangle_{\mbb R}=\left\|Q^{\frac{1}{2}}\lambda\right\|_{H^0}^2.
\end{align*}
Thus, $\Lambda(\lambda')=\frac{\alpha^2}{4}\left\|Q^{\frac{1}{2}}\lambda\right\|_{H^0}^2=\frac{\alpha^2}{4}\left\|\lambda'\right\|^2_{L^2(\mu)}$.
Recall that 
\begin{align*}
	\Lambda^*(x)=\underset{\lambda'\in (H^0)^*}{\sup}\left\{\lambda'(x)-\Lambda(\lambda')\right\}.
\end{align*}

For a given $x\in H^0$, if $\Lambda^*(x)<+\infty$, then there exists a constant $C(x)<+\infty$ such that $\lambda'(x)\leq\frac{\alpha^2}{4}\left\|\lambda'\right\|_{L^2(\mu)}^2+C(x)$. Define the linear functional $x^{**}$ on $\left((H^0)^*,\left\|\cdot\right\|_{L^2(\mu)}\right)\subseteq\overline{(H^0)^*}^{L^2(\mu)}$ by $x^{**}(\lambda')=\lambda'(x)$, for every $\lambda'\in(H^0)^*$. Then we have $\underset{\lambda'\in(H^0)^*,~\left\|\lambda'\right\|_{L^2(\mu)}\leq1}{\sup}x^{**}(\lambda')\leq\frac{\alpha^2}{4}+C(x)$. It means that $x^{**}$ is a bounded linear functional on $\left((H^0)^*,\left\|\cdot\right\|_{L^2(\mu)}\right)$. By Hahn--Banach theorem and the fact that  $\left((H^0)^*,\left\|\cdot\right\|_{L^2(\mu)}\right)$ is dense in $\overline{(H^0)^*}^{L^2(\mu)}$, $x^{**}$ can be uniquely extended to $\overline{(H^0)^*}^{L^2(\mu)}$. (In fact, for each $\lambda'\in\overline{(H^0)^*}^{L^2(\mu)}$, take $\lambda'_n\in (H^0)^*$ such that $\lambda'_n\to\lambda'$ in the norm $\left\|\cdot\right\|_{L^2(\mu)}$. Then the extended functional is $x^{**}(\lambda')=\lim_{n\to \infty}x^{**}(\lambda'_n)$.) The extended functional is still denoted by $x^{**}$. In this way, for every $x\in H^0$ satisfying $\Lambda^*(x)<+\infty$, we obtain a bounded linear functional on $\overline{(H^0)^*}^{L^2(\mu)}$ such that $x^{**}(\lambda')=\lambda'(x)$ for each $\lambda'\in (H^0)^*$.
By Riesz representation theorem, there exists some $h\in\overline{(H^0)^*}^{L^2(\mu)}$ such that 
$x^{**}(\lambda')=\left\langle \lambda',h\right\rangle_{L^2(\mu)}$ for each $\lambda'\in\overline{(H^0)^*}^{L^2(\mu)}$. Hence, $\lambda'(x)=\left\langle \lambda',h\right\rangle_{L^2(\mu)}$ for each $\lambda'\in(H^0)^*$. Further, we have that
\begin{align*}
	\lambda'(x)=\int_{H^0}h(z)\lambda'(z)\mu(dz)=\lambda'\left(\int_{H^0}zh(z)\mu(dz)\right)=\lambda'(J(h)),\qquad\forall\quad\lambda'\in(H^0)^*.
\end{align*}
By the arbitrariness of $\lambda'$, $x=J(h)$. Hence, $\Lambda^*(x)<+\infty$ implies that $x\in \mcal  H_{\mu}=J\left(\overline{(H^0)^*}^{L^2(\mu)}\right)=Im(Q^{\frac{1}{2}})$, where $Im(Q^{\frac{1}{2}})$ is the image of  $Q^{\frac{1}{2}}$.

On the other hand, if $H^0\ni x=J(h)$ for some $h\in\overline{(H^0)^*}^{L^2(\mu)}$, then
\begin{align*}
	\Lambda^*(x)&=\Lambda^*(J(h))=\underset{\lambda'\in (H^0)^*}{\sup}\left\{\lambda'(J(h))-\frac{\alpha^2}{4}\left\|\lambda'\right\|^2_{L^2(\mu)}\right\}\nonumber\\
	&=\underset{\lambda'\in (H^0)^*}{\sup}\left\{\left\langle \lambda',h\right\rangle_{L^2(\mu)}-\frac{\alpha^2}{4}\left\|\lambda'\right\|^2_{L^2(\mu)}\right\}.\nonumber
\end{align*}
Noting the continuity of $\left\langle \lambda',h\right\rangle_{L^2(\mu)}-\frac{\alpha^2}{4}\left\|\lambda'\right\|^2_{L^2(\mu)}$ with respect to $\lambda'$ in the norm $\left\|\cdot\right\|_{L^2(\mu)}$, and that $\left((H^0)^*,\left\|\cdot\right\|_{L^2(\mu)}\right)$ is dense in $\overline{(H^0)^*}^{L^2(\mu)}$, we have
\begin{align*}
\Lambda^*(x)
&=\underset{g\in\overline{(H^0)^*}^{L^2(\mu)}}{\sup}\left\{\left\langle g,h\right\rangle_{L^2(\mu)}-\frac{\alpha^2}{4}\left\|g\right\|^2_{L^2(\mu)}\right\}\nonumber\nonumber\\
&\leq\underset{g\in\overline{(H^0)^*}^{L^2(\mu)}}{\sup}\left\{\frac{1}{2}\left[\frac{\alpha^2}{2}\left\|g\right\|^2_{L^2(\mu)}+\frac{2}{\alpha^2}\left\|h\right\|^2_{L^2(\mu)}\right]-\frac{\alpha^2}{4}\left\|g\right\|^2_{L^2(\mu)}\right\}\nonumber\\
&=\frac{1}{\alpha^2}\left\|h\right\|^2_{L^2(\mu)}. \nonumber
\end{align*}
Taking $g=\frac{2}{\alpha^2}h$ leads to $\Lambda^*(x)\geq\frac{1}{\alpha^2}\left\|h\right\|^2_{L^2(\mu)}$. Thus, we obtain
\begin{align}\label{sec3k1}
	\Lambda^*(x)=\frac{1}{\alpha^2}\left\|h\right\|^2_{L^2(\mu)}=\frac{1}{\alpha^2}\left\|x\right\|_{\mcal H_{\mu}}^2=\frac{1}{\alpha^2}\left\|Q^{-\frac{1}{2}}x\right\|_{H^0}^2.
\end{align}
Finally we have
\begin{align*}
	\Lambda^*(x)=\begin{cases}
	\frac{1}{\alpha^2}\left\|Q^{-\frac{1}{2}}x\right\|_{H^0}^2, \qquad &\text{if}~x\in Q^{\frac{1}{2}}(H^0), \\
	+\infty, &\text{otherwise},
	\end{cases}
\end{align*}
which completes this proof.
\end{proof}
Similar to the proof of  \cite[Proposition 3.1]{Cohen}, we obtain $\mbf E\|u(T)\|^2_{H^0}=\mbf E\|u_0\|^2_{H^0}+\alpha^2T\tr(Q)$, where $\tr(Q)=\sum_{k=1}^{\infty}\eta_{k}$. Then, by  Markov's inequality, one has
	that for each $R>0$ and sufficiently large $T$
	\begin{align}\label{sec3tail}
		\mbf P\left(\left\|u(T)\right\|^2_{H^0}\geq T^2R^2\right)\leq\frac{\mbf E\|u(T)\|^2_{H^0}}{T^2R^2}\leq\frac{C}{T},
	\end{align}
	for some constant $C$ independent of $T$.
	In what follows, we show that the probability of the tail event of the mass  $\|u(T)\|^2_{H^0}$ in \eqref{sec3tail} can be exponentially small.
	More precisely, by Lemma \ref{contraction} and Theorem \ref{LDP for BT}, we immediately obtain the LDP of $\left\{\left\|B_T\right\|_{H^0}\right\}_{T>0}$, which yields the following corollary.
\begin{cor}\label{sec3cor}
	If $Q$ is an injection, then 
	it holds that
	\begin{itemize}
		\item[(1)] $\left\{\left\|B_T\right\|_{H^0}\right\}_{T>0}$ satisfies an LDP on $\mbb R^+:=[0,+\infty)$ with the good rate function
		\begin{align*}
			J(y)=\frac{1}{\alpha^2}\inf_{z\in H^0,~\left\|Q^{\frac{1}{2}}z\right\|_{H^0}=y}\|z\|^2_{H^0},\qquad y\geq0.
		\end{align*}
		\item[(2)] 	For every $R>0$ and $\varepsilon>0$, there is some $T_0$ such that
		\begin{align}\label{sec3mass}
		\mbf P\left(\left\|u(T)\right\|^2_{H^0}\geq T^2R^2\right)\leq \exp\left\{-T\left(\inf_{y\geq R}J(y)-\varepsilon\right)\right\},\qquad\forall\quad T\geq T_0,
		\end{align}
		and $\inf\limits_{y\geq R}J(y)\in(0,+\infty)$. 
	\end{itemize}
\end{cor}
\begin{proof}
	(1) Since the mapping $\|\cdot\|_{H^0}:H^0\to\mbb R^+$ is continuous, it follows from Lemma \ref{contraction} and Theorem \ref{LDP for BT} that $\left\{\left\|B_T\right\|_{H^0}\right\}_{T>0}$ satisfies an LDP on $\mbb R^+$ with the good rate function
	\begin{align*}
		J(y)=&\inf_{x\in H^0,~\|x\|_{H^0}=y}I(x)=\inf_{x\in Q^{\frac{1}{2}}(H^0),~\|x\|_{H^0}=y}I(x) \\
		=&\frac{1}{\alpha^2}\inf_{x\in Q^{\frac{1}{2}}(H^0),~\|x\|_{H^0}=y}\left\|Q^{-\frac{1}{2}}x\right\|^2_{H^0}\\
		=&\frac{1}{\alpha^2}\inf_{z\in H^0,~\left\|Q^{\frac{1}{2}}z\right\|_{H^0}=y}\|z\|^2_{H^0},
	\end{align*}
	where we have used the assumption that $Q$ is an injection. This proves the first conclusion.
	
	(2) Clearly, the set $\left\{z\in H^0,~\left\|Q^{\frac{1}{2}}z\right\|_{H^0}=y\right\}$ is nonempty for every $y\geq0$. Hence, $J(y)<+\infty$ for every $y\geq 0$. Accordingly, $\inf\limits_{y\geq R}J(y)<+\infty$ for each $R>0$. In addition, we claim $J(y)>0$ for each $y>0$. In fact, if for some $y_0>0$, $J(y_0)=0$, then there is a sequence $\{z_n\}_{n\in \mbb N}\subseteq H^0$ such that $\left\|Q^{\frac{1}{2}}z_n\right\|_{H^0}=y_0$ and $\lim\limits_{n\to\infty}\|z_n\|_{H^0}=0$. Noting that $Q^{\frac{1}{2}}$ is a continuous operator, then we have $y_0=\lim\limits_{n\to\infty}\left\|Q^{\frac{1}{2}}z_n\right\|_{H^0}=0$, which yields a contradiction. Hence, we prove the claim. Using the fact that  a good rate function can achieve its infimum on every nonempty closed set (see e.g.,\cite[Page 4]{Dembo}), we have that for each $R>0$, there is some $y_R\geq R$ such that $\inf\limits_{y\geq R}J(y)=J(y_R)>0$. It remains to prove \eqref{sec3mass}. Since $\left\{\left\|B_T\right\|_{H^0}\right\}_{T>0}$ satisfies the LDP with the rate function $J$, we obtain that for each fixed $R>0$,
	\begin{align*}
		\limsup_{T\to\infty}\frac{1}{T}\ln\mbf P\left(\frac{\|u(T)\|_{H^0}}{T}\geq R\right)\leq-\inf_{y\geq R}J(y).
	\end{align*}
The above formula implies that for every $\varepsilon>0$, there is a $T_0>0$ such that 
\begin{align*}
	\frac{1}{T}\ln\mbf P\left(\frac{\|u(T)\|_{H^0}}{T}\geq R\right)\leq-\inf_{y\geq R}J(y)+\varepsilon,\qquad \forall\quad T\geq T_0.
\end{align*}
Hence we have that
\begin{align*}
\mbf P\left(\|u(T)\|^2_{H^0}\geq T^2R^2\right)=\mbf P\left(\frac{\|u(T)\|_{H^0}}{T}\geq R\right)\leq \exp\left\{-T\left(\inf_{y\geq R}J(y)-\varepsilon\right)\right\},\qquad\forall\quad T\geq T_0.
\end{align*}
This completes the proof.
\end{proof}
\begin{rem}\label{sec3rem}
	For sufficiently large $L>0$, one can always find $R$ and $T$ such that $T^2R^2\leq L$. Then by \eqref{sec3rem} one has that  $\mbf P\left(\left\|u(T)\right\|^2_{H^0}\geq L\right)\leq\mbf P\left(\left\|u(T)\right\|^2_{H^0}\geq T^2R^2\right)\leq \exp\left\{-T\left(\inf_{y\geq R}J(y)-\varepsilon\right)\right\}$. This indicates that the probability of the tail event of the mass of \eqref{xde1} is exponentially small on a sufficiently large time.
\end{rem}

\section{LDP for the spatial spectral Galerkin approximation}\label{Sec4}
In the previous section, we derive the LDP of $\{B_T\}_{T>0}$ for the continuous system \eqref{xde1}. In order to obtain a valid approximation for the rate function $I$ of $\{B_T\}_{T>0}$, we apply the spatial spectral Galerkin method to \eqref{xde1}, and study the LDP of $\{B^M_T\}_{T>0}$ of spectral Galerkin approximation. Here, $B_T^M$ is a discrete approximation of $B_T$, which will be specified later. 

For $M\in \mbb N$, we define the finite dimensional subspace $H_M:=\text{span}\left\{e_1,e_2,\ldots,e_M\right\}$ of $(H^0,\langle\cdot,\cdot\rangle_{\mbb C})$ and the projection operator $P_M: H^0\to H_M$ by $P_M x=\sum_{k=1}^{M}\langle x,e_k\rangle_{\mbb C}e_k$ for each $x\in H^0$. 
Then $P_M$ is also a projection operator from $(U^0,\langle\cdot,\cdot\rangle_{\mbb R})$ onto $U_M$  such that $P_M x=\sum_{k=1}^{M}\langle x,e_k\rangle_{\mbb R}e_k$ for each $x\in U^0$. Denote $\Delta_M=\Delta P_M$. Using the above notations, we get the following spectral Galerkin approximation:
\begin{align}\label{Galerkin}
	d u^M(t)&=\bm{i}\Delta_M u^M(t)dt+\bm{i}\alpha P_Md W(t),\qquad t>0, \\
	u^M(0)&=P_Mu_0\in H_M.\nonumber
\end{align}
It is verified that \eqref{Galerkin} admits a unique mild solution on $H_M$ given by 
\begin{align}\label{sec4k1}
	u^M(t)=S_M(t)u^M(0)+\bm{i}\alpha\int_{0}^{t}S_M(t-s)P_Md W(s),
\end{align}
where $S_M(t)=e^{\bm{i}t\Delta_M}$ is the unitary $C_0$-group generated by $\bm{i}\Delta_M$. 

For the spatial discretization \eqref{Galerkin}, we define $B^M_T=\frac{u^M(T)}{T}$ which is viewed as a discrete approximation for $B_T$. In what follows, we study the LDP of $\{B^M_T\}_{T>0}$ and whether $\{u^M\}_{M\in\mbb N}$ can asymptotically preserve the LDP of $B_T$.
\subsection{LDP for $\{B^M_T\}_{T>0}$ }
Following the ideas of deriving the LDP of $\{B_T\}_{T>0}$, in this part, we give the LDP of $\{B^M_T\}_{T>0}$. For this end, we first consider the logarithmic moment generating function $\Lambda^M(\lambda)=\lim\limits_{T\to\infty}\frac{1}{T}\ln\mbf E\exp\left\{T\left\langle \lambda,B^M_T\right\rangle_\mbb R \right\}$, for each $\lambda\in H_M$. Then, we study the exponential tightness of $\{B^M_T\}_{T>0}$. Finally, by means of Theorem \ref{GE}, we obtain the LDP of $\{B^M_T\}_{T>0}$.   Hereafter we use the notation $K(a_1,\ldots,a_m)$ to denote some constant dependent on the parameters $a_1 ,\ldots,a_m$ but independent of  $T$ and $N$, which may vary from one line to another.
\begin{theo}\label{sec4tho4.1}
		For each fixed $M\in\mbb N$, $\left\{B^M_T\right\}_{T>0}$ satisfies an LDP on $H^0$ with  the good rate function $I^M(\cdot)$ given by 
	\begin{align} \label{sec4k2}
	I^M(x)=\begin{cases}
	\frac{1}{\alpha^2}\left\|Q_M^{-\frac{1}{2}}x\right\|_{H^0}^2, \qquad &\text{if}~x\in Q_M^{\frac{1}{2}}(H^0), \\
	+\infty, &\text{otherwise},
	\end{cases}
	\end{align}
	where $Q_M:=QP_M$ and $Q^{-\frac{1}{2}}_M$ is the pseudo inverse of  $Q_M^{\frac{1}{2}}$ on $H_M$, i.e., $Q_M^{-\frac{1}{2}}x=\underset{z}{argmin}\left\{\|z\|_{H^0}:\right.$ $\left.z\in H_M,Q^{\frac{1}{2}}_Mz=x\right\}$ for every $x\in H_M$.
\end{theo}
\begin{proof}
	Noting that $S_M(t)=\cos(t\Delta_M)+\bm{i}\sin(t\Delta_M)$, we have
	\begin{align}\label{sec4k3}
	u^M(T)&=S_M(T)u^M(0)-\alpha\int_{0}^{T}\sin((T-s)\Delta_M)P_Md W(s)+\bm{i}\alpha\int_{0}^{T}\cos((T-s)\Delta_M)P_Md W(s) \nonumber\\
	&=:S_M(T)u^M(0)-\alpha W^M_{\sin}(T)+\bm{i}\alpha W^M_{\cos}(T).
	\end{align}
	Notice that for each $T>0$, $W^M_{\sin}(T)$ is a  Gaussian random variable taking values on $(U_M,\langle\cdot,\cdot\rangle_\mbb R)$.
	By Proposition \ref{Cor}, the covariance operator $\mbf{Var}\left(W^M_{\sin}(T)\right)$ of $W^M_{\sin}(T)$ is
	\begin{align}\label{sec4k4}
		\mbf{Var}(W^M_{\sin}(T))&=\int_{0}^{T}\sin^2((T-s)\Delta_M)Q_Mds \nonumber\\
		&=\frac{Q_M}{2}\int_{0}^{T}\left[I-\cos(2(T-s)\Delta_M)\right]ds\nonumber\\
		&=\frac{TQ_M}{2}-\frac{Q_M\Delta_M^{-1}}{4}\sin(2T\Delta_M),
	\end{align}
	where $Q_M=QP_M$.
	Similarly, we have that
	\begin{align}\label{sec4k5}
		W^M_{\cos}(T)\sim\mcal N(0,\mbf{Var}(W^M_{\cos}(T)))\qquad \text{on}~U_M
	\end{align}
	with $\mbf{Var}(W^M_{\cos}(T))=\frac{TQ_M}{2}+\frac{1}{4}Q_M\Delta_M^{-1}\sin(2T\Delta_M)$. And the correlation operator $\mbf{Cor}\left(W^M_{\sin}(T),W^M_{\cos}(T)\right)$ is
	\begin{align} \label{sec4k6}
		\mbf{Cor}\left(W^M_{\sin}(T),W^M_{\cos}(T)\right)=\frac{Q_M\Delta_M^{-1}}{4}\left[I-\cos(2T\Delta_M)\right].
	\end{align}
For each $\lambda\in H_M$, we write it as $\lambda=\Re\lambda+\bm{i}\Im\lambda$ with $\Re\lambda$, $\Im\lambda\in U_M$. Then by \eqref{sec4k3}, 
\begin{align}\label{sec4k7}
	\langle u^M(T),\lambda\rangle_\mbb R=\langle S_M(T)u^M(0),\lambda\rangle_\mbb R-\alpha\left\langle W^M_{\sin}(T),\Re\lambda\right\rangle_\mbb R+\alpha\left\langle W^M_{\cos}(T),\Im\lambda\right\rangle_\mbb R.
\end{align}
Hence, we obtain
\begin{align}\label{sec4k8}
	\left|\mbf E\langle u^M(T),\lambda\rangle_\mbb R\right|=\left|\langle S_M(T)u^M(0),\lambda\rangle_\mbb R\right|\leq K(\lambda).
\end{align}
It follows from \eqref{sec4k4}, \eqref{sec4k5}, \eqref{sec4k6} and \eqref{sec4k7} that
\begin{align}\label{sec4k9}
	\mbf{Var}\langle u^M(T),\lambda\rangle_\mbb R=&\alpha^2\mbf{Var}\left\langle W^M_{\sin}(T),\Re\lambda\right\rangle_\mbb R+\alpha^2\mbf{Var}\left\langle W^M_{\cos}(T),\Im\lambda\right\rangle_\mbb R\nonumber\\
	&-2\alpha^2\mbf{Var}\left(\left\langle W^M_{\sin}(T),\Re\lambda\right\rangle_\mbb R,\left\langle W^M_{\cos}(T),\Im\lambda\right\rangle_\mbb R\right)\nonumber\\
	=&\alpha^2\left\langle \mbf{Var}(W^M_{\sin}(T))\Re\lambda,\Re\lambda\right\rangle_\mbb R+\alpha^2\left\langle \mbf{Var}(W^M_{\cos}(T))\Im\lambda,\Im\lambda\right\rangle_\mbb R\nonumber\\
	&-2\alpha^2\left\langle \mbf{Cor}(W^M_{\sin}(T),(W^M_{\cos}(T))\Re\lambda,\Im\lambda\right\rangle_\mbb R\nonumber\\
	=&\frac{\alpha^2T}{2}\left\langle Q_M\Re \lambda,\Re\lambda
	\right\rangle_\mbb R+\frac{\alpha^2T}{2}\left\langle Q_M\Im \lambda,\Im\lambda
	\right\rangle_\mbb R-\frac{\alpha^2}{4}\left\langle Q_M\Delta_M^{-1}\sin(2T\Delta_M)\Re \lambda,\Re\lambda
	\right\rangle_\mbb R\nonumber\\
	&+\frac{\alpha^2}{4}\left\langle Q_M\Delta_M^{-1}\sin(2T\Delta_M)\Im \lambda,\Im\lambda
	\right\rangle_\mbb R-\frac{\alpha^2}{2}\left\langle Q_M\Delta_M^{-1}\left(I-\cos(2T\Delta_M)\right)\Re \lambda,\Im\lambda
	\right\rangle_\mbb R\nonumber\\
	=:&\frac{\alpha^2T}{2}\left\langle Q_M\Re \lambda,\Re\lambda
	\right\rangle_\mbb R+\frac{\alpha^2T}{2}\left\langle Q_M\Im \lambda,\Im\lambda
	\right\rangle_\mbb R+R(T)
\end{align}
with
$|R(T)|\leq K(M,\lambda)$.
Using \eqref{sec4k8} and \eqref{sec4k9}, we have that, for every $\lambda\in H_M$,
\begin{align*}
	\Lambda^M(\lambda)=&\lim_{T\to \infty}\frac{1}{T}\ln\mbf Ee^{T\left\langle B^M_T,\lambda\right\rangle_\mbb R}=\lim_{T\to \infty}\frac{1}{T}\ln\mbf Ee^{\left\langle u^M(T),\lambda\right\rangle_\mbb R}\nonumber\\
	=&\lim_{T\to \infty}\frac{1}{T}\left(\mbf E\left\langle u^M(T),\lambda\right\rangle_\mbb R+\frac{1}{2}\mbf{Var}\left\langle u^M(T),\lambda\right\rangle_\mbb R\right)\nonumber\\
	=&\frac{1}{2}\left(\frac{\alpha^2}{2}\left\langle Q_M\Re \lambda,\Re\lambda
	\right\rangle_\mbb R+\frac{\alpha^2}{2}\left\langle Q_M\Im \lambda,\Im\lambda
	\right\rangle_\mbb R\right)\nonumber\\
	=&\frac{\alpha^2}{4}\left\|Q_M^{\frac{1}{2}}\lambda\right\|^2_{H^0}.\nonumber
\end{align*}

Analogous to the proof  of \eqref{sec3k1}, we have
\begin{align*}
	(\Lambda^M)^*(x)=\underset{\lambda\in H_M}{\sup}\left\{\langle\lambda,x\rangle-\Lambda^M(\lambda)\right\}=\begin{cases}
	\frac{1}{\alpha^2}\left\|Q_M^{-\frac{1}{2}}x\right\|_{H^0}^2, \qquad &\text{if}~x\in Q_M^{\frac{1}{2}}(H_M), \\
	+\infty &\text{otherwise}.
	\end{cases}
\end{align*}

Next, we show that $\left\{B^M_T\right\}_{T>0}$ is exponentially tight.  Define $K_L=\left\{f\in H_M\left|\right.\|f\|_{H^0}\leq L\right\}$, then $K_L$ is the compact subset of $H_M$. It follows from \eqref{sec4k3} that 
\begin{align}
&\mbf P\left(B^M_T\in K_L^c)\right)=\mbf P\left(\|u^M(T)\|_{H^0}>LT\right)\nonumber\\
\leq&\mbf P\left(\|S_M(T)u^M_0\|_{H^0}>\frac{TL}{3}\right)+\mbf P\left(\alpha\|W^M_{\sin}(T)\|_{U^0}>\frac{TL}{3}\right)+\mbf P\left(\alpha\|W^M_{\cos}(T)\|_{U^0}>\frac{TL}{3}\right). \label{sec4k10}
\end{align}
By \eqref{sec4k4}, we have
	\begin{align}\label{sec4k11}
\frac{W^M_{\sin}(T)}{\sqrt{T}}\sim\mcal N\left(0,\left(\frac{I}{2}-\frac{\Delta_M^{-1}\sin(2T\Delta_M)}{4T}\right)Q_M\right)\qquad \text{on}~U_M.
\end{align}
Hence, we obtain
\begin{align*}
	\lambda_k\left(\mbf{Var}\left(\frac{W^M_{\sin}(T)}{\sqrt{T}}\right)\right)=\left(\frac{1}{2}-\frac{\sin(2Tk^2)}{4Tk^2}\right)\eta_k=\frac{1}{2}\left(1-\frac{\sin(2Tk^2)}{2Tk^2}\right)\eta_k<\eta_k,\qquad k=1,2,\dots,M.
\end{align*}
For every $0<\varepsilon<\frac{1}{2\eta_1}$, it follows from Proposition \ref{Fernique} that
\begin{align*}
		\mbf E\exp\left\{{\varepsilon\left\|\frac{W^M_{\sin}(T)}{\sqrt{T}}\right\|_{U^0}^2}\right\}&=\left[\det\left(I-2\varepsilon\mbf{Var}\left(\frac{W^M_{\sin}(T)}{\sqrt{T}}\right)\right)\right]^{-\frac{1}{2}}<\left[\det(I-2\varepsilon Q_M)\right]^{-\frac{1}{2}}=C(\varepsilon,Q_M).
\end{align*}
The above formula yields 
\begin{align}\label{sec4k12}
	\mbf P\left(\alpha\|W^M_{\sin}(T)\|_{U^0}>\frac{TL}{3}\right)&=P\left(\exp\left\{\varepsilon\left\|\frac{W^M_{\sin}(T)}{\sqrt{T}}\right\|^2_{U^0}\right\}>\exp\left\{\frac{
		\varepsilon TL^2}{9\alpha^2}\right\}\right)\nonumber\\ 
	&\leq e^{-\frac{\varepsilon TL^2}{9\alpha^2}}\mbf E\exp\left\{{\varepsilon\left\|\frac{W^M_{\sin}(T)}{\sqrt{T}}\right\|_{U^0}^2}\right\}\leq e^{-\frac{\varepsilon TL^2}{9\alpha^2}}C(\varepsilon,Q_M).
\end{align}
Similarly, one has 
\begin{align}\label{sec4k13}
\mbf P\left(\alpha\|W^M_{\cos}(T)\|_{U^0}>\frac{TL}{3}\right)\leq e^{-\frac{\varepsilon TL^2}{9\alpha^2}}C(\varepsilon,Q_M).
\end{align}
According to Proposition \ref{limsup}, \eqref{sec4k12} and \eqref{sec4k13}, we have
\begin{align*}
	\limsup_{T\to\infty}\frac{1}{T}\ln\mbf P\left(B^M_T\in K_L^c\right)\leq-\frac{\varepsilon L^2}{9\alpha^2},\qquad 0<\varepsilon<\frac{1}{2\eta_1},
\end{align*}
where we have used the fact that $\mbf P\left(\|S_M(T)u^M_0\|_{H^0}>\frac{TL}{3}\right)=0$ for sufficiently large $T$. 
Then, we obtain 
\begin{align*}
	\lim_{L\to \infty}\limsup_{T\to\infty}\frac{1}{T}\ln\mbf P\left(B^M_T\in K_L^c\right)=-\infty,
\end{align*} 
which implies the exponential tightness of $\left\{B^M_T\right\}_{T>0}$.

Notice that $\Lambda^M(\cdot)$ is Fr\'echet differentiable and $\mcal D\Lambda^M(\lambda)(\cdot)=\frac{\alpha^2}{2}\left\langle Q_M\lambda,\cdot \right\rangle$ for each $\lambda\in H_M$. Then it follows from Theorem \ref{GE} that $\left\{B^M_T\right\}_{T>0}$ satisfies an LDP on $H_M$ with  the good rate function
\begin{align*}
	\widetilde I^M(x)=(\Lambda^M)^*(x)=\begin{cases}
	\frac{1}{\alpha^2}\left\|Q_M^{-\frac{1}{2}}x\right\|_{H^0}^2, \qquad &\text{if}~x\in Q_M^{\frac{1}{2}}(H_M), \\
	+\infty, &x\in H_M\setminus Q_M^{\frac{1}{2}}(H_M).
	\end{cases}
\end{align*}
Clearly, $H_M$ is the closed subspace of $H^0$ and for each $T>0$, $\mbf P(B^M_T\in H^M)=1$. Thus, using Lemma \ref{subLDP} and the fact $Q_M^{\frac{1}{2}}(H_M)=Q_M^{\frac{1}{2}}(H^0)$, we conclude that $\left\{B^M_T\right\}_{T>0}$ satisfies an LDP on $H^0$ with  the good rate function
\begin{align*}
I^M(x)=\begin{cases}
\frac{1}{\alpha^2}\left\|Q_M^{-\frac{1}{2}}x\right\|_{H^0}^2, \qquad &\text{if}~x\in Q_M^{\frac{1}{2}}(H^0), \\
+\infty, &\text{otherwise}.
\end{cases}
\end{align*}
\end{proof}

\subsection{Weakly asymptotical preservation for the LDP of $\{B_T\}_{T>0}$}
In the last subsection, we obtain the LDP for $\{B^M_T\}_{T>0}$ of the spectral Galerkin approximation $\{u^M(T)\}_{T>0}$. It is natural to consider whether $I^M$ converges to $I$ pointwise   as $M$ tends to infinity. In \cite{LDPosc}, authors give the definition of \emph{asymptotical preservation
	 for the LDP} of the original system, i.e., the discrete rate functions of numerical methods converge to that of the original system in the pointwise sense. In our case, since generally  $Q_M^{\frac{1}{2}}(H^0)\subsetneqq Q^{\frac{1}{2}}(H^0)$,  it can not be assured that $I^M$ converges to $I$ pointwise. However, the sequence  $\left\{Q_M^{\frac{1}{2}}(H^0)\right\}_{M\in\mbb N}$ of sets converges to $Q^{\frac{1}{2}}(H^0)$ by the fact $\lim\limits_{M\to\infty}Q_M^{\frac{1}{2}} x=Q^{\frac{1}{2}}x$ for each $x\in H^0$. It is hoped that $I^M$ is a good approximation of $I$ when $M$ is large enough. Thus, we give the following definition.

\begin{Def}\label{Def4.2}
 	For a spatial semi-discretization $\{u^M\}_{M\in\mbb N}$  of \eqref{xde1}, denote $B^M_T=\frac{u^M(T)}{T}$ . Assume that $\{B^M_T\}_{T>0}$  satisfies an LDP on $H^0$ with the rate function $I^M$ for all sufficiently large $M$.  Then we say that $\{u^M\}_{M\in\mbb N}$  weakly asymptotically preserves the LDP of $\{B_T\}_{T>0}$  if for each $x\in Q^{\frac{1}{2}}(H^0)$ and $\varepsilon>0$, there exist $x_0\in H^0$ and $M\in\mbb N$ such that
 	\begin{align}\label{wasym}
 	\left\|x-x_0\right\|_{H^0}<\varepsilon,\qquad \left|I(x)-I^M(x_0)\right|<\varepsilon,
 	\end{align}
 	where $I$ is the rate function of $\{B_T\}_{T>0}$.
 \end{Def}


\begin{theo}\label{sec4tho4.4}
	For the spectral Galerkin approximation \eqref{Galerkin}, $\{u^M\}_{M\in\mbb N}$ weakly asymptotically preserves the LDP of $\{B_T\}_{T>0}$, i.e., \eqref{wasym} holds.
\end{theo}
\begin{proof}
	This problem is discussed in the following two cases.
	
	\emph{Case 1: There are infinitely many $0$ in $\{\eta_k\}_{k\in\mbb N}$, i.e., for some $l\in \mbb N$, $\eta_{l}>\eta_{l+1}=\eta_{l+2}=\cdots=0$.} \\
	For this case, $Q$ degenerates to a finite-rank operator. If $M\geq l$, then $Q_M=Q$. Hence, it holds that $I^M(x)=I(x)$ for every $x\in H^0$, which implies \eqref{wasym}. We say that $\{u^M\}_{M\in\mbb N}$  exactly preserves the LDP of $\{B_T\}_{T>0}$ for this case (see \cite[Definition 4.1]{LDPosc}).
	
	\emph{Case 2: There are finitely many $0$ in $\{\eta_k\}_{k\in\mbb N}$.} \\
	Notice that for each finite $M\in\mbb N$, $\eta_{1}\geq\eta_2\geq\cdots\geq\eta_M>0$. 
	We denote $y=Q^{-\frac{1}{2}}x$ and
	define $x_M:=Q^{\frac{1}{2}}_My$. Further, we have
	\begin{align*}
		Q_M^{-\frac{1}{2}}x_M=&\underset{z}{argmin}\left\{\|z\|_{H^0}:~z\in H_M,~ Q_M^{\frac{1}{2}}z=x_M\right\}\\
		=&\underset{z}{argmin}\left\{\|z\|_{H^0}:~z\in  H_M,~ Q^{\frac{1}{2}}z=Q^{\frac{1}{2}}P_My\right\}\\
		=&\underset{z}{argmin}\left\{\|z\|_{H^0}:~z\in  H_M,~ \sqrt{\eta_k}\langle z,e_k\rangle_\mbb C=\sqrt{\eta_k}\langle y,e_k\rangle_\mbb C,~k=1,2\dots,M\right\}\\
	=&P_My.
	\end{align*}
The above formula yields
\begin{align}\label{sec4k14}
	\lim_{M\to\infty}\left|I^M(x_M)-I(x)\right|=\frac{1}{\alpha^2}\lim_{M\to\infty}\left|\|P_My\|^2_{H^0}-\|y\|^2_{H^0}\right|=0.
\end{align}
In addition, it holds that
\begin{align}\label{sec4k15}
	\lim_{M\to\infty}x_M=\lim_{M\to\infty}Q_M^{\frac{1}{2}}y=\lim_{M\to\infty}P_MQ^{\frac{1}{2}}y=Q^{\frac{1}{2}}y=x.
\end{align}
Thus, it follows from \eqref{sec4k14} and \eqref{sec4k15} that for each $x\in Q^{\frac{1}{2}}(H^0)$ and $\varepsilon>0$, there exist sufficiently large $M$ and $x_0=Q_M^{\frac{1}{2}}\left(Q^{-\frac{1}{2}}x\right)$ such that \eqref{wasym} holds. 

Combining Case 1 and Case 2, we complete the proof.
\end{proof}

\begin{rem}
	As is seen in  the proof of Theorem \ref{sec4tho4.4}, for every $x\in Q^{\frac{1}{2}}(H^0)$ and sufficiently large $M$, $I^M(Q_M^{\frac{1}{2}}Q^{-\frac{1}{2}}x)$ is a
	 good approximation of $I(x)$.  
\end{rem}

\section{LDP by spatio-temporal full discretization}\label{Sec5}
In this section, we investigate the LDP for the full discretizations, spatially by the  spectral Galerkin method and temporally by the symplectic methods or non-symplectic ones. We show that the full discretization weakly asymptotically preserves the LDP of $\{B_T\}_{T>0}$ when using a symplectic method in temporal  direction, while it does not share this property for a temporal non-symplectic method. These results indicate that the  modified rate function of the full discretization, based on the spatial spectral Galerkin method  and a temporal symplectic method, is a good approximation of $I$.    
\subsection{Full discretization}
Since the spectral Galerkin approximation $\{u^M(t)\}_{t\geq 0}$ takes values in $H_M$, it holds that $u^M(t)=\sum_{k=1}^{M}\left\langle u^M(t),e_k\right\rangle_\mbb Ce_k$. Denote $U^M(t)=\left(\left\langle u^M(t),e_1\right\rangle_\mbb C,\left\langle u^M(t),e_2\right\rangle_\mbb C\right.$ $,\left.\cdots,\left\langle u^M(t),e_M\right\rangle_\mbb C\right)^\top$. Let $U^{M,k}(t)$ be the $k$th component of $U^M(t)$. It follows from  $\eqref{Galerkin}$ that
\begin{align*}
	\ud U^{M,k}(t)=-\bm{i}k^2U^{M,k}(t)\ud t+\bm{i}\alpha\sqrt{\eta_k}\ud  \beta_k(t),\qquad k=1,2,\ldots,M.
\end{align*}
Then, we obtain a $\mbb C^M$-valued SDE
\begin{align*}
	\ud U^M(t)=-\bm{i}\mcal MU^M(t)\ud t+\bm{i}\alpha\mcal Q\ud \beta(t),
\end{align*} 
where $\mcal M=\text{diag}\left(1,2^2,\ldots,M^2\right)\in \mbb R^{M\times M}$, $\mcal Q=\text{diag}\left(\sqrt{\eta_{l}},\sqrt{\eta_{2}},\ldots,\sqrt{\eta_{M}}\right)\in \mbb R^{M\times M}$, and $\beta(t)=\left(\beta_1(t),\beta_2(t),\ldots,\beta_M(t)\right)^\top\in\mbb R^M$. Further,  using the notation $U^M(t)=P^M(t)+\bm{i}Q^M(t)$ with $P^M(t)=\Re U^M(t)$ and $Q^M(t)=\Im U^M(t)$, we obtain a $2M$-dimensional stochastic Hamiltonian system 
\begin{align}\label{SHS}
	\ud P^M(t)&=\mcal MQ^M(t)\ud t,\nonumber\\
	\ud Q^M(t)&=-\mcal MP^M(t)\ud t+\alpha\mcal Q\ud \beta(t),
\end{align}
which is equivalent to the system \eqref{Galerkin} with $\left\langle u^M(t),e_k\right\rangle_\mbb C=P^{M,k}(t)+\bm{i} Q^{M,k}(t)$, where $P^{M,k}$ and $Q^{M,k}$ are the $k$th arguments of $P^M$ and $Q^M$, respectively. Hence, in order to obtain the numerical method for \eqref{Galerkin}, we only need to consider discretizing the equivalent system \eqref{SHS}.

Denote by $\left\{(p^M_n,q^M_n)\right\}_{n\in\mbb N}$  the numerical approximation of $\left\{(P^M(t),Q^M(t))\right\}_{t\geq0}$.  Let
$F$ be the linear function from $\mbb C^M$ to $H^0$ defined by 
\begin{align}\label{sec5F}
F(z)=\sum_{k=1}^{M}z_ke_k,\qquad \forall \quad z=(z_1,z_2,\ldots,z_M)\in \mbb C^M.
\end{align}
Then we obtain the numerical solution $\{u^M_n\}_{n\in\mbb N}$ with $u^M_n:=F(p^M_n+\bm{i}q^M_n)$.
Further, we define $B^M_N=\frac{u^M_N}{N\tau}$  (see \cite{LDPosc}), where $\tau$ is the temporal stepsize. Then $B^M_N$ is a discrete approximation of $B_T$.  To give the LDP for $\{B^M_N\}_{N\in\mbb N}$, our idea is to first investigate the LDP of $\{A^M_N\}_{N\in\mbb N}$, where $A^M_N=\frac{p^M_N+\bm{i}q^M_N}{N\tau}.
$ Then noting that $B^M_N=F(A^M_N)$,
 combining the LDP of $\{A^M_N\}_{N\in\mbb N}$ on $\mbb C^M$ and the contraction principle (Lemma \ref{contraction}), we derive the LDP of $\{B^M_N\}_{N\in\mbb N}$. More precisely, we divide \eqref{SHS}  into the following $M$ subsystems
\begin{align}\label{subSHS}
	\ud \begin{pmatrix}
	P^{M,k}(t)\\
	Q^{M,k}(t)
	\end{pmatrix}=k^2
	\begin{pmatrix}
	0&1\\
	-1&0
	\end{pmatrix}
	\begin{pmatrix}
	P^{M,k}(t)\\
	Q^{M,k}(t)
	\end{pmatrix}\ud t+\alpha_k
	\begin{pmatrix}
	0\\
	1
	\end{pmatrix}\ud\beta_k(t),\qquad k=1,2,\ldots,M,
\end{align}  
where $\alpha_k=\alpha\sqrt{\eta_k}$, $k=1,2,\ldots,M$.
For each $k\in\{1,2,\ldots,M\}$,  we consider the general numerical method in the following form
\begin{align}\label{timemethod}
	\left(\begin{array}{c}
p^{M,k}_{n+1}\\\\
q^{M,k}_{n+1}
\end{array}\right)=
\left(
	\begin{array}{cc}
	a_{11}(k^2\tau)&a_{12}(k^2\tau)\\
	\\
	a_{21}(k^2\tau)&a_{22}(k^2\tau)
	\end{array}\right)
	\left(\begin{array}{cc}
	p^{M,k}_n\\\\
	q^{M,k}_n
	\end{array}\right)
	+\alpha_k
	\left(\begin{array}{cc}
	b_1(k^2\tau)\\\\
	b_2(k^2\tau)
	\end{array}\right)\delta \beta_{k,n},
\end{align}
where $\delta \beta_{k,n}=\beta_{k}(t_{n+1})-\beta_{k}(t_n)$ with $t_n=n\tau$, $n=1,2,\ldots$, and functions $a_{ij},\,b_i: (0,\infty)\to\mbb R$, $i,j=1,2$ are  determined by a concrete method. In addition, we require $b_1^2(h)+b_2^2(h)\neq 0$ for all sufficiently small $h$. Hence, we finally obtain the numerical solution $\left\{(p^M_n,q^M_n)\right\}_{n\in\mbb N}$ generated by \eqref{timemethod}, with $(p^{M,k}_n,q^{M,k}_n)$ being the $k$th component of $(p^M_n,q^M_n)$, $n=1,2,\dots$. By defining functions
\begin{align}
A(h):=\left(
\begin{array}{cc}
a_{11}(h)&a_{12}(h)\\
\\
a_{21}(h)&a_{22}(h)
\end{array}\right),\qquad B(h):=	\left(\begin{array}{cc}
b_1(h)\\\\
b_2(h)
\end{array}\right), \qquad\forall\quad h>0,
\end{align}
we rewrite \eqref{timemethod} as
\begin{align}\label{sec5recur}
\left(\begin{array}{c}
p^{M,k}_{n+1}\\\\
q^{M,k}_{n+1}
\end{array}\right)=
A(k^2\tau)\left(\begin{array}{cc}
p^{M,k}_n\\\\
q^{M,k}_n
\end{array}\right)
+\alpha_k
B(k^2\tau)\delta \beta_{k,n}, \qquad n=0,1,2\ldots
\end{align}
with $p^{M,k}_0+\bm{i}q^{M,k}_0=\left\langle u^M(0),e_k\right\rangle_\mbb C$.

Next we introduce some concrete temporal discretizations   taking the form \eqref{sec5recur}. 
\begin{ex}[Midpoint Scheme]
	Applying midpoint scheme to \eqref{Galerkin} yields
	\begin{align*}
	u^M_{n+1}=u^M_n+\frac{1}{2}\bm{i}\tau\Delta_M \left(u^M_{n}+u^M_{n+1}\right)+\bm{i}\alpha P_M\delta W_n,\qquad n=0,1,2\ldots,
	\end{align*}
with
	\begin{align*}
	A^1(h):=\frac{1}{4+h^2}\left(
	\begin{array}{cc}
	4-h^2&4h\\
	-4h&4-h^2
	\end{array}\right),\qquad B^1(h):=\frac{2}{4+h^2}	\left(\begin{array}{cc}
	h\\
	2
	\end{array}\right), \qquad\forall\quad h>0.
	\end{align*}
	Here $\delta W_n:=W(t_{n+1})-W(t_n)$.
\end{ex}
\begin{ex}[Exponential Euler Method] The exponential Euler method for \eqref{Galerkin} is
	\begin{align*}
	u^M_{n+1}=S_M(\tau)u^M_n+\bm{i}\alpha S_M(\tau)P_M\delta W_{n},\qquad n=0,1,2\ldots,
	\end{align*}
	with
	\begin{align*}
	A^2(h):=\left(
	\begin{array}{cc}
	\cos(h)&\sin(h)\\
	-\sin(h)&\cos(h)
	\end{array}\right),\qquad B^2(h):=	\left(\begin{array}{cc}
	\sin(h)\\
	\cos(h)
	\end{array}\right), \qquad\forall\quad h>0.
	\end{align*}
\end{ex}
\begin{ex}[Backward Euler--Maruyama Method] The backward Euler--Maruyama method for \eqref{Galerkin} reads
	\begin{align*}
	u^M_{n+1}=u^M_n+\bm{i}\tau\Delta_M u^M_{n+1}+\bm{i}\alpha P_M\delta W_n,\qquad n=0,1,2\ldots,
	\end{align*}
	 with
	\begin{align*}
	A^3(h):=\frac{1}{1+h^2}\left(
	\begin{array}{cc}
	1&h\\
	-h&1
	\end{array}\right),\qquad B^3(h):=\frac{1}{1+h^2}	\left(\begin{array}{cc}
	h\\
	1
	\end{array}\right), \qquad\forall\quad h>0.
	\end{align*}
\end{ex}

Next, we give our main assumptions on functions $A$ and $B$, which will be used to derive the LDP of $\left\{B^M_N\right\}_{N\in\mbb N}$.
\begin{assum}\label{assum1}
There is some $h_1>0$ such that
\begin{align*}
4\det(A(h))-({\rm tr}(A(h)))^2>0,\qquad \forall \quad h<h_1,
\end{align*}
where ${\rm tr}(A)$ and $\det(A)$ denote the trace and the determinant of $A$, respectively.	
\end{assum} 
We will use \textbf{Assumption \ref{assum1}} to give the general expression of the method \eqref{timemethod}, following the idea of \cite{LDPosc}.
\begin{assum}\label{assum2}
There is some $h_2>0$ such that for all $h<h_2$, $\det(A(h))=1$.
\end{assum}
\noindent One can show that  the numerical  method  generated by \eqref{timemethod} is symplectic if and only if \textbf{Assumption \ref{assum2}} holds. In fact,  $\left\{(p^M_n,q^M_n)\right\}_{n\in\mbb N}$ generated by \eqref{timemethod} is symplectic for all sufficiently small $\tau>0$ if and only if for all sufficiently small $\tau>0$,
$dp^M_{n+1}\wedge dq^M_{n+1}=dp^M_{n}\wedge dq^M_{n}$, i.e.,
\begin{align*}
\sum_{k=1}^{M}dp^{M,k}_{n+1}\wedge dq^{M,k}_{n+1}=\sum_{k=1}^{M}dp^{M,k}_{n}\wedge dq^{M,k}_{n},\qquad n=1,2,\ldots
\end{align*}
According to \eqref{timemethod}, it holds that $dp^{M,k}_{n+1}\wedge dq^{M,k}_{n+1}=\left(a_{11}(k^2\tau)a_{22}(k^2\tau)-a_{12}(k^2\tau)a_{21}(k^2\tau)\right)dp^{M,k}_{n}\wedge dq^{M,k}_{n}$. Hence, the
method  generated by \eqref{timemethod} is symplectic for all sufficiently small $\tau>0$ if and only if for all sufficiently small $\tau>0$, $k=1,2,\ldots,M$,
\begin{align*}
a_{11}(k^2\tau)a_{22}(k^2\tau)-a_{12}(k^2\tau)a_{21}(k^2\tau)=1,	
\end{align*}
which is equivalent to that there is some $h_0>0$ such that
\begin{align*}
a_{11}(h)a_{22}(h)-a_{12}(h)a_{21}(h)=1,\qquad \forall\quad h<h_0,
\end{align*}
i.e., \textbf{Assumption \ref{assum2}} holds.

\begin{assum}\label{assum3}
There exist some  $\eta\in(0,1)$ and some $h_3>0$ such that
\begin{align*}
|c(h)|<(1-\eta)\sqrt{a(h)b(h)}, \qquad\forall\quad h<h_3.
\end{align*}
Here, functions $a$, $b$, $c:(0,\infty)\to\mbb R$ are defined by 
\begin{align*}
a=&(a_{11}b_1+a_{12}b_2-b_1)^2+b_1(a_{11}b_1+a_{12}b_2)(2-{\rm tr}(A)),\\
b=&(a_{21}b_1-a_{11}b_2+b_2)^2-b_2(a_{21}b_1-a_{11}b_2)(2-{\rm tr}(A)),\\
	c=&\frac{1}{2}\left(a_{21}b_1-a_{11}b_2\right)b_1{\rm tr}(A)+b_1b_2\left(\frac{1}{2}({\rm tr}(A))^2-1\right)\nonumber\\
&-\left(a_{11}b_1+a_{12}b_2\right)\left(a_{21}b_1-a_{11}b_2\right)
-\frac{1}{2}{\rm tr}(A)\left(a_{11}b_1+a_{12}b_2\right)b_2.
\end{align*} 
\end{assum}
	\textbf{Assumption \ref{assum3}} is used to give the explicit expression of the rate functions of  $\{A^M_N\}_{N\in\mbb N}$ and $\{B^M_N\}_{N\in\mbb N}$.  In fact,  $a(h),b(h)>0$ for sufficiently small $h$,  whose proof is similar to those of Lemmas 3.2 and 5.1 in \cite{LDPosc}. In addition, we have the following property.
\begin{rem}\label{sec5eq1}
	Under \textbf{Assumption \ref{assum2}}, $c=\frac{a_{11}-a_{22}}{2}\left[a_{12}b_2^2-a_{21}b_1^2+b_1b_2\left(a_{11}-a_{22}\right)\right].$
\end{rem}
This is because under \textbf{Assumption \ref{assum2}}, $\det(A)=a_{11}a_{22}-a_{12}a_{21}=1$. Then it follows that
\begin{align*}
c=&b_1^2\left(\frac{1}{2}a_{21}\tr(A)-a_{11}a_{21}\right)+b_2^2\left(a_{11}a_{12}-\frac{1}{2}a_{12}\tr(A)\right)\nonumber\\
&+b_1b_2\left[\frac{1}{2}(\tr(A))^2-1-(a_{12}a_{21}-a_{11}^2)-a_{11}\tr(A)\right]\nonumber\\
=&b_1^2\left(\frac{1}{2}a_{21}\tr(A)-a_{11}a_{21}\right)+b_2^2\left(a_{11}a_{12}-\frac{1}{2}a_{12}\tr(A)\right)\nonumber\\
&+b_1b_2\left[\frac{1}{2}(a_{11}-a_{22})^2+a_{11}a_{22}-a_{12}a_{21}-1\right]\nonumber\\
=&\frac{a_{11}-a_{22}}{2}\left[a_{12}b_2^2-a_{21}b_1^2+b_1b_2\left(a_{11}-a_{22}\right)\right].
\end{align*}
When we investigate the LDP of $\{B^M_N\}_{N\in\mbb N}$ via temporal non-symplectic methods, we give the following assumption (see \cite{LDPosc}).
\begin{assum}\label{assum4}
	There is some $h_4>0$ such that for all $h<h_4$, $\det(A(h))<1$.
\end{assum}
In addition, when investigating the asymptotical preservation of $\left\{u^M_n\right\}_{M,n\in\mbb N}$ for the LDP of $\left\{B_T\right\}_{T>0}$, we give the following assumption concerning the convergence of the numerical method. 
\begin{assum}\label{assum5}
$
\left|a_{11}-1\right|+\left|a_{22}-1\right|+\left|a_{12}-h\right|+\left|a_{21}+h\right|=\mcal{O}(h^2),~\text{and}~\left|b_1\right|+\left|b_2-1\right|=\mcal O(h).$
\end{assum}
\noindent One can prove that under  \textbf{Assumption \ref{assum5}}, $\left\{\left(p^M_n,q^M_n\right)\right\}_{n\in\mbb N}$ corresponding to \eqref{timemethod} has at least first order convergence in mean-square sense. For more details, one  refers to \cite{LDPosc}.

It is verified that the methods  in Examples $1$ and $2$ are symplectic satisfying  \textbf{Assumptions \ref{assum1}-\ref{assum3} and \ref{assum5}}.  And the method in Example $3$ is non-symplectic satisfying \textbf{Assumptions \ref{assum1} and \ref{assum4}}.

To characterize the asymptotical preservation of $\{u^M_n\}_{M,n\in\mbb N}$ for the LDP of $\{B_T\}_{T>0}$, we give the following definition (see \cite{LDPosc} for the similar definition).
\begin{Def}\label{sec5def5.4}
 	For a spatio-temporal full discretization $\{u^M_n\}_{M,n\in\mbb N}$ of \eqref{xde1} with temporal stepsize $\tau$, denote $B^M_N=\frac{u^M_N}{N\tau}$. Assume that for each fixed $M\in\mbb N$, $\{B^M_N\}_{N\in\mbb N}$  satisfies an LDP on $H^0$ with the rate function $I^{M,\tau}$. We call $I^{M,\tau}_{mod}:=\frac{I^{M,\tau}}{\tau}$  the modified rate function. Then $\{u^M_n\}_{M,n\in\mbb N}$  is said to weakly asymptotically preserves the LDP of $\{B_T\}_{T>0}$  if for each $x\in Q^{\frac{1}{2}}(H^0)$ and $\varepsilon>0$, there exist $x_0\in H^0$, $M>0$ and $\tau>0$ such that
 	\begin{align}\label{wasym1}
 	\left\|x-x_0\right\|_{H^0}<\varepsilon,\qquad \left|I(x)-I^{M,\tau}_{mod}(x_0)\right|<\varepsilon.
 	\end{align}
\end{Def}
With the above preparation, we give our main results of this paper. That is, for the full discretization $\{u^M_n\}_{M,n\in\mbb N}$ with $u^M_n=F(p^{M}_n+\bm{i}q^M_n)$, where $\{p^{M}_n,q^M_n\}_{M,n\in\mbb N}$ is the numerical solution corresponding to \eqref{timemethod}, when the temporal discretization is symplectic, it weakly asymptotically preserves the LDP of $\{B_T\}_{T>0}$, while it does not possess this property for a temporal non-symplectic discretization. 
\begin{theo}\label{sec5tho5.5}
 If \textbf{Assumptions \ref{assum1}, \ref{assum2} and \ref{assum5}}  hold, then  
	\begin{itemize}
	\item[(1)] For each fixed $M\in\mbb N$ with $\eta_{M}>0$,  we have that for all   sufficiently  small stepsize $\tau$,  $\{B^M_N\}_{N\in\mbb N}$ satisfies an LDP on $H^0$ with the good rate function given by \begin{align}\label{sec5IMtau}
	&I^{M,\tau}(x)\nonumber\\
	=&\begin{cases}
	\sum_{k=1}^{M}\frac{\tau\left(4-(\rm{tr}(A(k^2\tau)))^2\right)}{4\left[a(k^2\tau)b(k^2\tau)-c^2(k^2\tau)\right]\alpha_k^2}\left[b(k^2\tau)(\Re\left\langle x,e_k\right\rangle_{\mbb C})^2+a(k^2\tau)(\Im \left\langle x,e_k\right\rangle_{\mbb C})^2\right.\\
	\phantom{\sum_{k=1}^{M}\frac{\tau\left(4-(\rm{tr}(A))^2\right)}{4\left[a(k^2\tau)b(k^2\tau)-c^2(k^2\tau)\right]\alpha_k^2}}+2c(k^2\tau)\Re\left\langle x,e_k\right\rangle_{\mbb C}\Im\left\langle x,e_k\right\rangle_{\mbb C}\big], \qquad &\text{if}~x\in H_M, \\
	+\infty, &\text{otherwise}.
	\end{cases}
	\end{align} 
		
	\item[(2)] For each fixed $M\in\mbb N$ with $\eta_M>0$, $\{u^M_N\}_{N\in\mbb N}$ asymptotically preserves the LDP of $\{B^M_T\}_{T>0}$, i.e., the modified rate function satisfies
	\begin{align}\label{sec5k19}
	\lim_{\tau\to 0}I^{M,\tau}_{mod}(x)
	=I^M(x).
	\end{align} 
	\item[(3)] $\{u^M_n\}_{M,n\in\mbb N}$ weakly asymptotically preserves the LDP for $\{B_T\}_{T>0}$ of \eqref{xde1}, i.e.,
	\eqref{wasym1} holds.
	\end{itemize}
\end{theo}

\begin{theo}\label{sec5tho5.7}
	 If \textbf{Assumptions \ref{assum1} and \ref{assum4}}   hold, then for each $M\in\mbb N$, 
	$\{B^M_N\}_{N\in\mbb N}$ satisfies an LDP on $H^0$ with the good rate function
	\begin{align*}
	I_{ns}^{M,\tau}(x)=\begin{cases}
	0, \qquad &\text{if}~x=0, \\
	+\infty, &\text{otherwise}.
	\end{cases}
	\end{align*}
	Moreover,  $\{u^M_n\}_{M,n\in\mbb N}$ can not weakly asymptotically preserve the LDP for $\{B_T\}_{T>0}$ of \eqref{xde1}, i.e.,
	\eqref{wasym1} does not hold.
\end{theo}
For the method in Example $1$, we have that for each $M\in\mbb N$, there is sufficiently small $\tau$, 
$\{B^M_N\}_{N\in\mbb N}$ satisfies an LDP on $H^0$ with the good rate function
\begin{align*}
I_1^{M,\tau}(x)=\begin{cases}
\frac{\tau}{\alpha^2}\left\|\left(I+\frac{\tau^2\Delta_M^2}{4}\right)Q_M^{-\frac{1}{2}}x\right\|, \qquad &\text{if}~x\in Q_M^{\frac{1}{2}}(H_0), \\
+\infty, &\text{otherwise}.
\end{cases}
\end{align*}
Hence, $\lim\limits_{\tau\to 0}I_{1,mod}^{M,\tau}(x):=\lim\limits_{\tau\to 0} I_1^{M,\tau}(x)/\tau=I^M(x)$ for each $x\in H^0$.  In addition, for each  $h>0$, \textbf{Assumptions \ref{assum1} and \ref{assum2}}  hold. These
verify that midpoint scheme satisfies both the conditions and the first conclusion of Theorem \ref{sec5tho5.5}. Finally, combining Theorem \ref{sec4tho4.4}, we have that the full discretization spatially, by a spatial Galerkin method and  temporally by the midpoint scheme, weakly asymptotically the LDP of $\{B_T\}_{T>0}$.\\

Next we give the rate functions of $\{B^M_N\}_{N\in\mbb N}$ when using the methods in Examples $1$-$3$.
\begin{itemize}
	\item Midpoint Scheme\\
	The rate function of $\{B^M_N\}_{N\in\mbb N}$ is
	\begin{align*}
	I_1^{M,\tau}(x)=\begin{cases}
	\frac{\tau}{\alpha^2}\left\|\left(I+\frac{\tau^2\Delta_M^2}{4}\right)Q_M^{-\frac{1}{2}}x\right\|, \qquad &\text{if}~x\in Q_M^{\frac{1}{2}}(H_0), \\
	+\infty, &\text{otherwise}.
	\end{cases}
	\end{align*}
	\item Exponential Euler Method\\
	 The rate function of $\{B^M_N\}_{N\in\mbb N}$  is
	 \begin{align*}
	 I_2^{M,\tau}(x)=\begin{cases}
	 \frac{\tau}{\alpha^2}\left\|Q_M^{-\frac{1}{2}}x\right\|, \qquad &\text{if}~x\in Q_M^{\frac{1}{2}}(H_0), \\
	 +\infty, &\text{otherwise}.
	 \end{cases}
	 \end{align*}
	 In particular, we note that if $Q$ is a finite rank operator, i.e., there is $l\in\mbb N$ such $\eta_{l+1}=\eta_{l+2}=\cdots=0$, then $I^{l,\tau}_{2,mod}=I$. This indicates when noise takes values in finite dimensional space, this full discretization  preserves exactly the LDP of $\{B_T\}_{T>0}$. 
	 
	 \item Backward Euler--Maruyama Method\\
	 The rate function of $\{B^M_N\}_{N\in\mbb N}$ is
	 \begin{align*}
	 I_3^{M,\tau}(x)=\begin{cases}
	 0, \qquad &\text{if}~x=0, \\
	 +\infty, &\text{otherwise}.
	 \end{cases}
	 \end{align*}
\end{itemize}

\subsection{Proof of Theorem \ref{sec5tho5.5}}
In this part, we consider the LDP of $\{B^M_N\}_{N\in\mbb N}$ for the full discretizations of \eqref{xde1}, spatially by the spectral Galerkin method \eqref{Galerkin} and temporally by symplectic methods. To this end, we let \textbf{Assumption \ref{assum2}} hold throughout this part. Firstly, for every fixed $k\in \{1,2,\ldots,M\}$, we
derive the limit $\Lambda_k(z):=\lim\limits_{N\to \infty}\frac{1}{N}\ln\mbf E\exp\left\{\frac{1}{\tau}\left\langle z,p^{M,k}_N+\bm{i}q^{M,k}_N\right\rangle_{\mbb R}\right\}$ for $z\in\mbb C$, to give the expression of the logarithmic moment generating function  $\Lambda^{M,\tau}(\bm{\lambda})=\lim\limits_{N\to \infty}\frac{1}{N}\ln\mbf E\exp\left\{N\left\langle  \bm{\lambda},A^M_N \right\rangle_\mbb R\right\}$ of $\left\{A^M_N\right\}_{N\in\mbb N}$. Then using Theorem \ref{GE}, we obtain the LDP of  $\{A^M_N\}_{N\in\mbb N}$ for symplectic methods. Further,  the contraction principle (Lemma \ref{contraction}) leads to the LDP of $\{B^M_N\}_{N\in\mbb N}$ with $B^M_N:=F(A^M_N)$. Finally combining the convergence condition (\textbf{Assumption \ref{assum5}}), we prove that $\{u^M_n\}_{M,n\in\mbb N}$ weakly asymptotically preserves the LDP of $\{B_T\}_{T>0}$, which
completes the proof of Theorem \ref{sec5tho5.5}.
\begin{lem}\label{sec5lem1}
	 If \textbf{Assumptions \ref{assum1} and \ref{assum2}} hold, then for each fixed $M\in\mbb N$, we have that for all   sufficiently  small stepsize $\tau$, 
	 \begin{align} \label{sec5k11}
	 \Lambda^{M,\tau}(\bm{\lambda})=\sum_{k=1}^{M}\Lambda_k(\lambda_k)=\sum_{k=1}^{M}\frac{\alpha_k^2}{4\tau\sin^2(\theta_k)}\left\{a(k^2\tau)(\Re\lambda_k)^2+b(k^2\tau)(\Im\lambda_k)^2-2c(k^2\tau)\Re\lambda_k\Im\lambda_k\right\},
	 \end{align}
	 where $a,b,c$ are given in \textbf{Assumption \ref{assum3}}.
	 Moreover, $\Lambda^{M,\tau}$ is finite valued and Gateaux differentiable. 
\end{lem}
\begin{proof}
For each $\bm{\lambda}=\left(\lambda_1,\lambda_2,\ldots,\lambda_M\right)\in\mbb C^M$, we have
\begin{align*}
\left\langle \bm{\lambda},p^M_N+\bm{i}q^M_N\right\rangle_\mbb R=\sum_{k=1}^{M}\left\langle \lambda_k,p^{M,k}_N+\bm{i}q^{M,k}_N \right\rangle_\mbb R=\sum_{k=1}^{M}\left(\Re\lambda_kp^{M,k}_N+\Im\lambda_kq^{M,k}_N\right).
\end{align*}
Thus, the logarithmic moment generating function for $\{A^M_N\}_{N\in\mbb N}$ is 
\begin{align}\label{sec5k3}
\Lambda^{M,\tau}(\bm{\lambda})&=\lim_{N\to \infty}\frac{1}{N}\ln\mbf E\exp\left\{N\left\langle  \bm{\lambda},A^M_N \right\rangle_\mbb R\right\}\nonumber\\
&=\lim_{N\to \infty}\frac{1}{N}\ln\mbf E\exp\left\{\frac{1}{\tau}\left\langle  \bm{\lambda},p^M_N+\bm{i}q^M_N \right\rangle_\mbb R\right\}\nonumber\\
&=\lim_{N\to \infty}\frac{1}{N}\ln\mbf E\exp\left\{\sum_{k=1}^{M}\frac{1}{\tau}\left(\Re\lambda_kp^{M,k}_N+\Im\lambda_kq^{M,k}_N\right)\right\}\nonumber\\
&=\lim_{N\to \infty}\frac{1}{N}\ln\prod_{k=1}^{M}\mbf E\exp\left\{\frac{1}{\tau}\left(\Re\lambda_kp^{M,k}_N+\Im\lambda_kq^{M,k}_N\right)\right\}\nonumber\\
&=\sum_{k=1}^{M}\lim_{N\to \infty}\frac{1}{N}\ln\mbf E\exp\left\{\frac{1}{\tau}\left(\Re\lambda_kp^{M,k}_N+\Im\lambda_kq^{M,k}_N\right)\right\}\nonumber\\
&=\sum_{k=1}^{M}\Lambda_k(\lambda_k),
\end{align}
where we have used the fact that $\left\{\left(p^{M,k}_n,q^{M,k}_n\right)\right\}_{n\in\mbb N}$, $k=1,2,\ldots,M$, are mutually independent stochastic processes as a result of the independence of $\{\beta_{k}(t)\}_{t\geq0}$, $k=1,2\ldots,M$. 

To acquire the expression $\Lambda_k(\cdot
)$, we need to give the general formula of $\left\{\left(p^{M,k}_n,q^{M,k}_n\right)\right\}_{n\in\mbb N}$ (Hereafter, we always fix some $k\in\{1,2,\ldots,M\}$ without extra statement). 
It follows from the  recurrence formula \eqref{sec5recur} that
\begin{align*}
\left(\begin{array}{c}
p^{M,k}_{n}\\\\
q^{M,k}_{n}
\end{array}\right)=
\left(A(k^2\tau)\right)^{n}\left(\begin{array}{cc}
p^{M,k}_0\\\\
q^{M,k}_0
\end{array}\right)
+\alpha_k
\sum_{j=0}^{n-1}\left(A(k^2\tau)\right)^{n-j-1}B(k^2\tau)\delta \beta_{k,j}, \qquad n=0,1,2\ldots
\end{align*}
Let $\theta_k\in(0,\pi)$ be the parameter such that   
\begin{equation}\label{theta}
\cos(\theta_k)=\frac{\text{tr}(A(k^2\tau))}{2\sqrt{\det(A(k^2\tau))}},\qquad\sin(\theta_k)=\frac{\sqrt{4\det(A(k^2\tau))-(\text{tr}(A(k^2\tau)))^2}}{2\sqrt{\det(A(k^2\tau))}}.
\end{equation}
Then under \textbf{Assumption \ref{assum1}}, one has (also see \cite[Sect. 3]{LDPosc}) that for sufficiently small $\tau$,
\begin{align*}
	\left(A(k^2\tau)\right)^{n}=\left(
	\begin{array}{cc}
	-\det(A(k^2\tau))\hat{\alpha}^k_{n-1}+a_{11}(k^2\tau)\hat{\alpha}^k_{n}&a_{12}(k^2\tau)\hat{\alpha}^k_{n}\\
	\\
	a_{21}(k^2\tau)\hat{\alpha}^k_{n}&\hat{\alpha}^k_{n+1}-a_{11}(k^2\tau)\hat{\alpha}^k_{n}
	\end{array}\right),
\end{align*}
where $\hat{\alpha}^k_{n}=\left[\det(A(k^2\tau))\right]^{\frac{n-1}{2}}\sin(n\theta_k)/\sin(\theta_k)$.
In this way, we obtain the following expression of the general formula of $\left\{\left(p^{M,k}_n,q^{M,k}_n\right)\right\}_{n\in\mbb N}$
\begin{align}\label{sec5kp}
	p^{M,k}_n=&-\det(A)\hat{\alpha}^k_{n-1}p^{M,k}_0+\hat{\alpha}^k_{n}\left(a_{11}p^{M,k}_0+a_{12}q^{M,k}_0\right)\nonumber\\
	&+\alpha_k\sum_{j=0}^{n-1}\left[-\det(A)\hat{\alpha}^k_{n-2-j}b_1+(a_{11}b_1+a_{12}b_2)\hat{\alpha}^k_{n-1-j}\right]\delta\beta_{k,j}
\end{align}
and
\begin{align}\label{sec5kq}
q^{M,k}_n=&a_{21}\hat{\alpha}^k_{n}p^{M,k}_0+\hat{\alpha}^k_{n+1}q^{M,k}_0-a_{11}\hat{\alpha}^k_{n}q^{M,k}_0\nonumber\\
&+\alpha_k\sum_{j=0}^{n-1}\left[(a_{21}b_1-a_{11}b_2)\hat{\alpha}^k_{n-1-j}+b_2\hat{\alpha}^k_{n-j}\right]\delta\beta_{k,j},
\end{align}
where $\det(A)$, $a_{ij},b_i$, $i,j=1,2$, are computed at $k^2\tau$.
For convenience, when no confusion occurs, we always omit the argument  $k^2\tau$ of $\det(A)$, $a_{ij},b_i$, $i,j=1,2$.

Since \textbf{Assumption \ref{assum2}} holds,
\begin{equation}\label{symtheta}
\cos(\theta_k)=\frac{\text{tr}(A(k^2\tau))}{2},\qquad\sin(\theta_k)=\frac{\sqrt{4-(\text{tr}(A(k^2\tau)))^2}}{2},\qquad \hat{\alpha}^k_n=\frac{\sin(n\theta_k)}{\sin(\theta_k)}.
\end{equation}
It follows from \eqref{sec5kp}, \eqref{sec5kq} and \eqref{symtheta} that
\begin{align}\label{sec5Ep}
	\mbf Ep^{M,k}_N=&-\hat{\alpha}^k_{N-1}p^{M,k}_0+\hat{\alpha}^k_{N}\left(a_{11}p^{M,k}_0+a_{12}q^{M,k}_0\right)\nonumber\\
	=&\frac{1}{\sin(\theta_k)}\left[-\sin((N-1)\theta_k)p^{M,k}_0+\sin(N\theta_k)\left(a_{11}p^{M,k}_0+a_{12}q^{M,k}_0\right)\right]
\end{align}
and
\begin{align}\label{sec5Eq}
\mbf Eq^{M,k}_N=&\hat{\alpha}^k_{N+1}q^{M,k}_0+\hat{\alpha}^k_{N}\left(a_{21}p^{M,k}_0-a_{11}q^{M,k}_0\right)\nonumber\\
=&\frac{1}{\sin(\theta_k)}\left[\sin((N+1)\theta_k)q^{M,k}_0+\sin(N\theta_k)\left(a_{21}p^{M,k}_0-a_{11}q^{M,k}_0\right)\right].
\end{align}
In addition, we obtain
\begin{align*}
	\mbf {Var}(p^{M,k}_N)=&\tau\alpha_k^2\sum_{j=0}^{N-1}\left[-\hat\alpha^k_{N-2-j}b_1+(a_{11}b_1+a_{12}b_2)\hat\alpha^k_{N-1-j}\right]^2\nonumber\\
	=&\frac{\tau\alpha_k^2}{\sin^2(\theta_k)}\sum_{j=0}^{N-1}\left[b_1^2\sin^2((j-1)\theta_k)+(a_{11}b_1+a_{12}b_2)^2\sin^2(j\theta_k)\right.\nonumber\\
	&\left.-2(a_{11}b_1+a_{12}b_2)b_1\sin(j\theta_k)\sin((j-1)\theta_k)\right].
\end{align*}
Using the fact $2\sin(\alpha)\sin(\beta)=\cos(\alpha-\beta)-\cos(\alpha+\beta)$, we have
\begin{align}\label{sec5varp}
	\mbf {Var}(p^{M,k}_N)=&\frac{\tau\alpha_k^2N}{2\sin^2(\theta_k)}\left[b_1^2+(a_{11}b_1+a_{12}b_2)^2-2(a_{11}b_1+a_{12}b_2)b_1\cos(\theta_k)\right]+R_1(k),
\end{align}
where
\begin{align*}
	R_1(k)=\frac{\tau\alpha_k^2}{\sin^2(\theta_k)}\sum_{j=0}^{N-1}\left[-\frac{b_1^2}{2}\cos(2(j-1)\theta_k)-\frac{(a_{11}b_1+a_{12}b_2)^2}{2}\cos(2j\theta_k)+(a_{11}b_1+a_{12}b_2)b_1\cos((2j-1)\theta_k)\right].
\end{align*} 
By the facts $\sum_{n=1}^{N}\cos((2n+1)\theta)=\frac{\sin((2N+2)\theta)-\sin(2\theta)}{2\sin(\theta)}$ and $\sum_{n=1}^{N}\cos((2n)\theta)=\frac{\sin((2N+1)\theta)-\sin(\theta)}{2\sin(\theta)}$,\\ 
 $\left|\sum_{j=0}^{N-1}\cos(2(j-1)\theta_k)\right|+\left|\sum_{j=0}^{N-1}\cos(2j\theta_k)\right|\leq K(\tau,M)$ (Recall that we use the notation $K(\tau,M)$ to denote the constant dependent on $\tau,M$, but independent of $N$). Hence, we obtain $|R_1|\leq K(\tau,M)$.
Similarly, one has
\begin{align}\label{sec5varq}
	\mbf{Var}(q^{M,k}_N)=&\frac{\tau\alpha_k^2N}{2\sin^2(\theta_k)}\left[b_2^2+(a_{21}b_1-a_{11}b_2)^2+2(a_{21}b_1-a_{11}b_2)b_2\cos(\theta_k)\right]+R_2
\end{align}
with $|R_2|\leq K(\tau,M)$, and
\begin{align}\label{sec5Cor}
\mbf{Cor}(p^{M,k}_N,q^{M,k}_N)=-&\frac{\tau\alpha_k^2N}{2\sin^2(\theta_k)}\left[(a_{21}b_1-a_{11}b_2)b_1\cos(\theta_k)+b_1b_2\cos(2\theta_k)\right.\nonumber\\
&\left.-(a_{11}b_1+a_{12}b_2)(a_{21}b_1-a_{11}b_2)-(a_{11}b_1+a_{12}b_2)b_2\cos(\theta_k)\right]+R_3
\end{align}
with $|R_3|\leq K(\tau,M)$.
It follows from \eqref{sec5Ep} and \eqref{sec5Eq} that
\begin{align}\label{estimateE}
	\left|\mbf E\left\langle\lambda_k,p^{M,k}_N+\bm{i}q^{M,k}_N\right\rangle_{\mbb R}\right|=\left|\Re\lambda_k\mbf E p^{M,k}_N+\Im \lambda_k\mbf Eq^{M,k}_N\right|\leq K(\tau,M,\lambda_k).
\end{align}
Further, \eqref{sec5varp}, \eqref{sec5varq} and \eqref{sec5Cor} give
\begin{align}\label{estimate}
	&\mbf{Var}\left\langle\lambda_k,p^{M,k}_N+\bm{i}q^{M,k}_N\right\rangle_{\mbb R}\nonumber\\
	=&(\Re\lambda_k)^2\mbf{Var}(p^{M,k}_N)+(\Im\lambda_k)^2\mbf{Var}(q^{M,k}_N)+2\Re\lambda_k\Im\lambda_k\mbf{Cor}\left(p^{M,k}_N,q^{M,k}_N\right)\nonumber\\
	=&\frac{\tau\alpha_k^2N(\Re\lambda_k)^2}{2\sin^2(\theta_k)}\left[b_1^2+(a_{11}b_1+a_{12}b_2)^2-2(a_{11}b_1+a_{12}b_2)b_1\cos(\theta_k)\right]\nonumber\\
	&+\frac{\tau\alpha_k^2N(\Im\lambda_k)^2}{2\sin^2(\theta_k)}\left[b_2^2+(a_{21}b_1-a_{11}b_2)^2+2(a_{21}b_1-a_{11}b_2)b_2\cos(\theta_k)\right]\nonumber\\
	&-\frac{\Re\lambda_k\Im\lambda_k\tau\alpha_k^2N}{\sin^2(\theta_k)}\left[(a_{21}b_1-a_{11}b_2)b_1\cos(\theta_k)+b_1b_2\cos(2\theta_k)\right.\nonumber\\
	&\left.-(a_{11}b_1+a_{12}b_2)(a_{21}b_1-a_{11}b_2)-(a_{11}b_1+a_{12}b_2)b_2\cos(\theta_k)\right]+R
\end{align}
with $|R|\leq K(\tau,M,\lambda_k)$. Noting $\left\langle\lambda_k,p^{M,k}_N+\bm{i}q^{M,k}_N\right\rangle_{\mbb R}$ is Gaussian, we have that for each $\lambda_k\in\mbb C$,
\begin{align}\label{sec5lambdak}
	\Lambda_k(\lambda_k)&=\lim_{N\to \infty}\frac{1}{N}\ln\mbf E\exp\left\{\frac{1}{\tau}\left\langle \lambda_k,p^{M,k}_N+\bm{i}q^{M,k}_N\right\rangle_{\mbb R}\right\}\nonumber\\
	&=\lim_{N\to \infty}\frac{1}{N}\left(\frac{1}{\tau}\mbf E\left\langle \lambda_k,p^{M,k}_N+\bm{i}q^{M,k}_N\right\rangle_{\mbb R}+\frac{1}{2\tau^2}\mbf{Var}\left\langle \lambda_k,p^{M,k}_N+\bm{i}q^{M,k}_N\right\rangle_{\mbb R}\right)\nonumber\\
	&=\frac{\alpha_k^2(\Re\lambda_k)^2}{4\tau\sin^2(\theta_k)}\left[b_1^2+(a_{11}b_1+a_{12}b_2)^2-2(a_{11}b_1+a_{12}b_2)b_1\cos(\theta_k)\right]\nonumber\\
	&\phantom{=}+\frac{\alpha_k^2(\Im\lambda_k)^2}{4\tau\sin^2(\theta_k)}\left[b_2^2+(a_{21}b_1-a_{11}b_2)^2+2(a_{21}b_1-a_{11}b_2)b_2\cos(\theta_k)\right]\nonumber\\
	&\phantom{=}-\frac{\Re\lambda_k\Im\lambda_k\alpha_k^2}{2\tau\sin^2(\theta_k)}\left[(a_{21}b_1-a_{11}b_2)b_1\cos(\theta_k)+b_1b_2(2\cos(\theta_k)^2-1)\right.\nonumber\\
	&\phantom{=}\left.-(a_{11}b_1+a_{12}b_2)(a_{21}b_1-a_{11}b_2)-(a_{11}b_1+a_{12}b_2)b_2\cos(\theta_k)\right].
\end{align} 
Then, noting that $\tr(A(k^2\tau))=2\cos(\theta_k)$, we rewrite \eqref{sec5lambdak} as
\begin{align}\label{sec5lam}
	\Lambda_k(\lambda_k)=\frac{\alpha_k^2}{4\tau\sin^2(\theta_k)}\left\{a(k^2\tau)(\Re\lambda_k)^2+b(k^2\tau)(\Im\lambda_k)^2-2c(k^2\tau)\Re\lambda_k\Im\lambda_k\right\}.
\end{align}
By \eqref{sec5k3}, we get the expression \eqref{sec5k11}.

In addition, for each $\bm{\lambda}$, $z\in\mbb C^M$, the Gateaux derivative of $\Lambda^{M,\tau}$ is given by
\begin{align*}
	\mcal G\Lambda^{M,\tau}(\bm{\lambda})(z)=\sum_{k=1}^{M}\frac{\alpha_k^2}{4\tau\sin^2(\theta_k)}\left[2a(k^2\tau)\Re\lambda_k\Re z_k+2b(k^2\tau)\Im\lambda_k\Im z_k-2c(k^2\tau)\left(\Re\lambda_k\Im z_k+\Im\lambda_k\Re z_k\right)\right].
\end{align*}
\end{proof} 

According to Theorem \ref{GE}, in order to give the LDP of $\{A^M_N\}_{N\in\mbb N}$, it remains to show that $\{A^M_N\}_{N\in\mbb N}$ is exponentially tight. As is mentioned in Section \ref{Sec3}, we will use the finiteness of logarithmic moment generating function to derive the exponential tightness. In fact, we have the following lemma.
\begin{lem}\label{sec5lem5.2}
	If \textbf{Assumptions \ref{assum1} and \ref{assum2}} hold, then for each fixed $M\in\mbb N$, we have that for all   sufficiently  small stepsize $\tau$, $\{A^M_N\}_{N\in\mbb N}$ satisfies an LDP with the good rate function
	$(\Lambda^{M,\tau})^*(z)=\sup\limits_{\bm{\lambda}\in\mbb C^M}\left\{\left\langle\bm{\lambda},z\right\rangle_{\mbb R}-\Lambda^{M,\tau}(\bm{\lambda})\right\}$.
\end{lem}
\begin{proof}
It follows from Lemma \ref{sec5lem1} that for each $\bm{\lambda}\in\mbb C^M$,
\begin{align}\label{sec5k6}
	\Lambda^{M,\tau}(\bm{\lambda})=\lim\limits_{N\to \infty}\frac{1}{N}\ln\mbf E\exp\left\{N\left\langle  \bm{\lambda},A^M_N \right\rangle_\mbb R\right\}<\infty.
\end{align} 
In particular, we take $\bm{\lambda}=\left(0,\ldots,0,1,0,\ldots,0\right)$ in \eqref{sec5k6} with $1$ being its $k$th component. Then we obtain
\begin{align}\label{zetak1}
	\zeta_{k,1}:=\lim\limits_{N\to \infty}\frac{1}{N}\ln\mbf E\exp\left\{N \Re A^{M,k}_N\right\}<\infty,
\end{align}
where $A^{M,k}_N$ is the $k$th argument of $A^M_N$. Taking $\bm{\lambda}=\left(0,\ldots,0,-1,0,\ldots,0\right)$ in \eqref{sec5k6} with $-1$ being its $k$th component  yields
\begin{align}\label{zetak2}
	\zeta_{k,2}:=\lim\limits_{N\to \infty}\frac{1}{N}\ln\mbf E\exp\left\{-N\Re A^{M,k}_N\right\}<\infty.
\end{align}
For each $L>0$, using Markov's inequality one has
\begin{align*}
	\mbf P\left(\Re A^{M,k}_N>\frac{L}{2M}\right)=\mbf P\left(\exp\left\{N\Re A^{M,k}_N\right\}>\exp\left\{\frac{NL}{2M}\right\}\right)\leq\exp\left\{-\frac{NL}{2M}\right\}\mbf E\exp\left\{N\Re A^{M,k}_N\right\}
\end{align*}
and
\begin{align*}
\mbf P\left(\Re A^{M,k}_N<-\frac{L}{2M}\right)=\mbf P\left(\exp\left\{-N\Re A^{M,k}_N\right\}>\exp\left\{\frac{NL}{2M}\right\}\right)\leq\exp\left\{-\frac{NL}{2M}\right\}\mbf E\exp\left\{-N\Re A^{M,k}_N\right\}.
\end{align*}

Hence, \eqref{zetak1} leads to
\begin{align*}
	\limsup_{N\to \infty}\frac{1}{N}\ln\mbf P\left(\Re A^{M,k}_N>\frac{L}{2M}\right)\leq-\frac{L}{2M}+\zeta_{k,1},
\end{align*}
and \eqref{zetak2} leads to
\begin{align*}
\limsup_{N\to \infty}\frac{1}{N}\ln\mbf P\left(\Re A^{M,k}_N<-\frac{L}{2M}\right)\leq-\frac{L}{2M}+\zeta_{k,2}.
\end{align*}
Combining the above formulas and Proposition \ref{limsup}, we have
\begin{align}\label{sec5k7}
\limsup_{N\to \infty}\frac{1}{N}\ln\mbf P\left(|\Re A^{M,k}_N|>\frac{L}{2M}\right)\leq \max\left\{-\frac{L}{2M}+\zeta_{k,1},-\frac{L}{2M}+\zeta_{k,2}\right\}=-\frac{L}{2M}+\zeta'_k,
\end{align}
with $\zeta'_k=\max\{\zeta_{k,1},\zeta_{k,2}\}$.
By taking $\bm{\lambda}=\left(0,\ldots,0,\bm{i},0,\ldots,0\right)$ (resp. $\bm{\lambda}=\left(0,\ldots,0,-\bm{i},0,\ldots,0\right)$) in \eqref{sec5k6} with $\bm{i}$ (resp. $-\bm{i}$) being its $k$th component, and repeating the above procedure, we have
\begin{align}\label{sec5k8}
	\limsup_{N\to \infty}\frac{1}{N}\ln\mbf P\left(|\Im A^{M,k}_N|>\frac{L}{2M}\right)\leq -\frac{L}{2M}+\zeta''_k,
\end{align}
for some $\zeta''_k<\infty$. 

Further, it holds that for every $k=1,2,\ldots,M$,
\begin{align*}
	\mbf P\left(\left\| A^{M,k}_N\right\|>\frac{L}{M}\right)\leq\mbf P\left(|\Re A^{M,k}_N|>\frac{L}{2M}\right)+\mbf P\left(|\Im A^{M,k}_N|>\frac{L}{2M}\right),
\end{align*}
which together with \eqref{sec5k7}, \eqref{sec5k8} and Proposition \ref{limsup} yields
\begin{align}\label{sec5k9}
	\limsup_{N\to \infty}\frac{1}{N}\ln\mbf P\left(\left\| A^{M,k}_N\right\|>\frac{L}{M}\right)\leq-\frac{L}{2M}+\zeta_k,
\end{align}
with $\zeta_k=\max\{\zeta'_k,\zeta''_k\}$.
For $L>0$, define $K_L=\left\{z\in \mbb C^M:~\left\|z\right\|\leq L\right\}$, which is a compact subset of $\mbb C^M$. Then it holds that 
\begin{align}\label{sec5k10}
	\mbf P\left(A^M_N\in K_L^c\right)&=\mbf P\left(\left\| A^{M}_N\right\|>L\right)\leq\mbf P\left(\sum_{k=1}^{M}\left\| A^{M,k}_N\right\|>L\right)\leq\mbf P\left(\bigcup_{k=1}^{M}\left\{\left\| A^{M,k}_N\right\|>\frac{L}{M}\right\}\right)\nonumber\\
  &\leq\sum_{k=1}^{M}\mbf P\left(\left\| A^{M,k}_N\right\|>\frac{L}{M}\right).
\end{align}
Substituting \eqref{sec5k9} into \eqref{sec5k10} and using Proposition \ref{limsup}, one has
\begin{align*}
	\limsup_{N\to \infty}\frac{1}{N}\ln\mbf P\left( A^{M}_N\in K_L^c\right)\leq-\frac{L}{2M}+\max_{k=1,2,\ldots,M}\zeta_k.
\end{align*}
Then, one immediately has
\begin{align*}
	\lim_{L\to \infty}\limsup_{N\to \infty}\frac{1}{N}\ln\mbf P\left( A^{M}_N\in K_L^c\right)=-\infty,
\end{align*}
which implies the exponential tightness of $\{A^M_N\}_{N\in\mbb N}$. By Lemma \ref{sec5lem1}, the exponential tightness of $\{A^M_N\}_{N\in\mbb N}$ and Theorem \ref{GE}, we complete the proof.
\end{proof}
\begin{lem}\label{sec5tho5.3}
Let \textbf{Assumptions \ref{assum1}, \ref{assum2} and \ref{assum3}} hold.  For each fixed $M\in\mbb N$ with $\eta_{M}>0$,  we have that for all   sufficiently  small stepsize $\tau$, 
	\begin{align}\label{sec5k13}
	(\Lambda^{M,\tau})^*(z)=\sum_{k=1}^{M}\frac{\tau\left(4-(\rm{tr}(A(k^2\tau)))^2\right)}{4\left[a(k^2\tau)b(k^2\tau)-c^2(k^2\tau)\right]\alpha_k^2}\left[b(k^2\tau)(\Re z_k)^2+a(k^2\tau)(\Im z_k)^2+2c(k^2\tau)\Re z_k\Im z_k\right].
	\end{align}
\end{lem}
\begin{proof}
It follows from  \eqref{sec5k11} that the Fenchel--Legendre transform of $\Lambda^{M,\tau}$ is
\begin{align}\label{sec5lambda*}
	(\Lambda^{M,\tau})^*(z)=&\sup_{\bm{\lambda}\in\mbb C^M}\left\{\left\langle\bm{\lambda},z\right\rangle_{\mbb R}-\Lambda^{M,\tau}(\bm{\lambda})\right\}\nonumber\\
	=&\sup_{\lambda_1\in\mbb C}\sup_{\lambda_2\in\mbb C}\cdots\sup_{\lambda_M\in\mbb C}\left\{\sum_{k=1}^{M}\left\langle\lambda_k,z_k\right\rangle_{\mbb R}-\Lambda_k(\lambda_k)\right\}\nonumber\\
	=&\sum_{k=1}^{M}\sup_{\lambda_k\in\mbb C}\left\{\left\langle\lambda_k,z_k\right\rangle_{\mbb R}-\Lambda_k(\lambda_k)\right\}
	=:\sum_{k=1}^{M}\Lambda_k^*(z_k).
\end{align}
According to \eqref{sec5lam},
\begin{align*}
&\Lambda_k^*(z_k)
	=\sup_{\lambda_k\in\mbb C}\left\{\left\langle\lambda_k,z_k\right\rangle_{\mbb R}-\Lambda_k(\lambda_k)\right\}\\
	=&\sup_{\left(\Re\lambda_k,\Im\lambda_k\right)\in\mbb R^2}\left\{\Re\lambda_k\Re z_k+\Im\lambda_k\Im z_k-\frac{\alpha_k^2}{4\tau\sin^2(\theta_k)}\left[a(k^2\tau)(\Re\lambda_k)^2+b(k^2\tau)(\Im\lambda_k)^2-2c(k^2\tau)\Re\lambda_k\Im\lambda_k\right]\right\}\\
	=&\sup_{\left(x,y\right)\in\mbb R^2}\left\{(\Re z_k)x+(\Im z_k)y-\frac{\alpha_k^2}{4\tau\sin^2(\theta_k)}\left[a(k^2\tau)x^2+b(k^2\tau)y^2-2c(k^2\tau)xy\right]\right\}\\
	=&:\sup_{\left(x,y\right)\in\mbb R^2}f_k(x,y).
\end{align*}

Under \textbf{Assumption \ref{assum3}}, if $\tau$ is sufficiently small, then for each $k=1,2,\ldots,M$, $x,y\in\mbb R$,
\begin{align*}
	\left|\frac{2c(k^2\tau)xy}{a(k^2\tau)x^2+b(k^2\tau)y^2}\right|<(1-\eta)\frac{2\sqrt{a(k^2\tau)b(k^2\tau)}|xy|}{a(k^2\tau)x^2+b(k^2\tau)y^2}\leq(1-\eta)\frac{a(k^2\tau)x^2+b(k^2\tau)y^2}{a(k^2\tau)x^2+b(k^2\tau)y^2}=1-\eta,
\end{align*}
which implies $1-\frac{2c(k^2\tau)xy}{a(k^2\tau)x^2+b(k^2\tau)y^2}>\eta$ for every $x$, $y\in\mbb R$.
Then, we have
\begin{align*}
	\lim_{(x,y)\to\infty}f_k(x,y)=&\lim_{(x,y)\to\infty}\left(a(k^2\tau)x^2+b(k^2\tau)y^2\right)\left\{\frac{(\Re z_k)x+(\Im z_k)y}{a(k^2\tau)x^2+b(k^2\tau)y^2}\right.\nonumber\\
	&\left.-\frac{\alpha_k^2}{4\tau\sin^2(\theta_k)}\left[1-\frac{2c(k^2\tau)xy}{a(k^2\tau)x^2+b(k^2\tau)y^2}\right]\right\}=-\infty,
\end{align*}
which along with the continuity of $f_k$, implies that there exist $x_k,y_k$ satisfying $-\infty<x_k,y_k<+\infty$ such that $\sup_{\left(x,y\right)\in\mbb R^2}f_k(x,y)=f_k(x_k,y_k)$. Then, it holds that
\begin{align*}
\frac{\partial f_k(x_k,y_k)}{\partial x}&=\Re z_k-\frac{\alpha^2_k}{4\tau\sin^2(\theta_k)}\left[2a(k^2\tau)x_k-2c(k^2\tau)y_k\right]=0,\\
\frac{\partial f_k(x_k,y_k)}{\partial y}&=\Im z_k-\frac{\alpha^2_k}{4\tau\sin^2(\theta_k)}\left[2b(k^2\tau)y_k-2c(k^2\tau)x_k\right]=0.
\end{align*}

For a given $M\in\mbb N$ with $\eta_M>0$,  $\alpha_k=\alpha\sqrt{\eta_k}>0$, $k=1,2,\ldots,M$. 
Then, we obtain
\begin{align*}
	x_k&=\frac{2\tau\sin^2(\theta_k)\left(\Re z_k b(k^2\tau)+\Im z_k c(k^2\tau)\right)}{\left[a(k^2\tau)b(k^2\tau)-c^2(k^2\tau)\right]\alpha_k^2},\\
	y_k&=\frac{2\tau\sin^2(\theta_k)\left(\Im z_k a(k^2\tau)+\Re z_k c(k^2\tau)\right)}{\left[a(k^2\tau)b(k^2\tau)-c^2(k^2\tau)\right]\alpha_k^2},
\end{align*}
which leads to
\begin{align*}
	\Lambda^*_k(z_k)=&\Re z_k\frac{2\tau\sin^2(\theta_k)\left(\Re z_k b+\Im z_k c\right)}{\left(ab-c^2\right)\alpha_k^2}+\Im z_k \frac{2\tau\sin^2(\theta_k)\left(\Im z_k a+\Re z_k c\right)}{\left(ab-c^2\right)\alpha_k^2} \nonumber\\
	&-\frac{\alpha_k^2}{4\tau\sin^2(\theta_k)}\left[a\frac{4\tau^2\sin^4(\theta_k)\left(\Re z_k b+\Im z_k c\right)^2}{\left(ab-c^2\right)^2\alpha_k^4}+b\frac{4\tau^2\sin^4(\theta_k)\left(\Im z_k a+\Re z_k c\right)^2}{\left(ab-c^2\right)^2\alpha_k^4}\right.\\
	&\left.-2c\frac{4\tau^2\sin^4(\theta_k)\left(\Re z_k b+\Im z_k c\right)\left(\Im z_k a+\Re z_k c\right)}{\left(ab-c^2\right)^2\alpha_k^4}\right]\nonumber\\
	=&\frac{2\tau\sin^2(\theta_k)}{\left(ab-c^2\right)\alpha_k^2}\left[b(\Re z_k)^2+2c\Re z_k\Im z_k+a(\Im z_k)^2\right]\nonumber-\frac{\tau\sin^2(\theta_k)}{\left(ab-c^2\right)^2\alpha_k^2}\left[\right. a(\Re z_k b+\Im z_k c)^2\nonumber\\
	&+b(\Im z_k a+\Re z_k c)^2-2c(\Re z_k b+\Im z_k c)(\Im z_k a+\Re z_k c)\left.\right].
\end{align*}
Direct computations give
\begin{align*}
	 &a(\Re z_k b+\Im z_k c)^2+b(\Im z_k a+\Re z_k c)^2-2c(\Re z_k b+\Im z_k c)(\Im z_k a+\Re z_k c)\nonumber\\
	 =&\left(ab-c^2\right)\left[b(\Re z_k)^2+2c\Re z_k\Im z_k+a(\Im z_k)^2\right].
\end{align*}
In this way,   we have
\begin{align}\label{sec5k12}
	\Lambda_k^*(z_k)=\frac{\tau\sin^2(\theta_k)}{\left[a(k^2\tau)b(k^2\tau)-c^2(k^2\tau)\right]\alpha_k^2}\left[b(k^2\tau)(\Re z_k)^2+a(k^2\tau)(\Im z_k)^2+2c(k^2\tau)\Re z_k\Im z_k\right].
\end{align}
By  \eqref{symtheta}, \eqref{sec5lambda*} and \eqref{sec5k12}, 
we complete the proof.
\end{proof}

Now we  give the proof of Theorem \ref{sec5tho5.5}.

\vspace{2mm}
\textit{Proof of Theorem \ref{sec5tho5.5}}: 

(1) Clearly, $F$ is a continuous mapping from $\mbb C^M$ to $H^0$ (see \eqref{sec5F}).
By  Lemmas \ref{contraction} and  \ref{sec5lem5.2}, we deduce that $\{B^M_N\}_{N\in\mbb N}$, with $B^M_N=F(A^M_N)$, satisfies an LDP on $H^0$ with the good rate function 
\begin{align*}
	&I^{M,\tau}(x)=(\Lambda^{M,\tau})^*(F^{-1}(x))\nonumber\\
	=&\begin{cases}
	\sum_{k=1}^{M}\frac{\tau\left(4-(\tr(A(k^2\tau)))^2\right)}{4\left[a(k^2\tau)b(k^2\tau)-c^2(k^2\tau)\right]\alpha_k^2}\left[b(k^2\tau)(\Re\left\langle x,e_k\right\rangle_{\mbb C})^2+a(k^2\tau)(\Im \left\langle x,e_k\right\rangle_{\mbb C})^2\right.&\\
	\phantom{\sum_{k=1}^{M}\frac{\tau\left(4-(\tr(A))^2\right)}{4\left[a(k^2\tau)b(k^2\tau)-c^2(k^2\tau)\right]\alpha_k^2}}+2c(k^2\tau)\Re\left\langle x,e_k\right\rangle_{\mbb C}\Im\left\langle x,e_k\right\rangle_{\mbb C}\big], \qquad &\text{if}~x\in H_M, \\
	+\infty, &\text{otherwise}.
	\end{cases}
\end{align*}

(2) Denote $J_{mod}^{M,\tau}(z)=\frac{(\Lambda^{M,\tau}(z))^*}{\tau}$.  Then $I^{M,\tau}_{mod}(x)=J_{mod}^{M,\tau}(F^{-1}(x))$. It follows from \textbf{Assumptions \ref{assum2} and \ref{assum5}} that $a_{12}\sim h$, $a_{21}\sim -h$ and $2-\tr(A)=1+a_{11}a_{22}-a_{12}a_{21}-a_{11}-a_{22}=(a_{11}-1)(a_{22}-1)-a_{12}a_{21}\sim h^2$. Hence $4-(\tr(A))^2=(2+\tr(A))(2-\tr(A))\sim 4h^2$. In addition, it holds that $a_{11}b_1+a_{12}b_2-b_1=(a_{11}-1)b_1+a_{12}b_2\sim h$. These imply $a\sim h^2.$ Further, $a_{21}b_1-a_{11}b_2+b_2=\mcal O(h^2)$, $a_{21}b_1b_2\left(2-\tr(A)\right)=\mcal O(h^4)$, $a_{11}b_2^2\left(2-\tr(A)\right)\sim h^2$, and hence $b\sim h^2$. 
Similarly, we have $c=\mcal O(h^3)$, which leads to $ab-c^2\sim h^4$. 
These implies that under  \textbf{Assumptions \ref{assum5}},  \textbf{Assumptions \ref{assum3}} holds. 
Accordingly, it follows from \eqref{sec5k13} that for each $z\in \mbb C^M$
\begin{align}\label{sec5k17}
	\lim_{\tau\to 0}	J^{\tau}_{mod}(z)=&\lim_{\tau\to 0}\sum_{k=1}^{M}\frac{\left(4-(\rm{tr}(A))^2\right)}{4\left[a(k^2\tau)b(k^2\tau)-c^2(k^2\tau)\right]\alpha_k^2}\left[b(k^2\tau)(\Re z_k)^2+a(k^2\tau)(\Im z_k)^2+2c(k^2\tau)\Re z_k\Im z_k\right]\nonumber\\
	=&\sum_{k=1}^{M}\lim_{\tau\to 0}\frac{4(k^2\tau)^4\left((\Re z_k)^2+(\Im z_k)^2\right)+\mcal O(\tau^5)}{4(k^2\tau)^4\alpha_k^2}=\sum_{k=1}^{M}\frac{\|z_k\|^2}{\alpha_k^2}.
\end{align}
Hence,
\begin{align*}
\lim_{\tau\to 0}I^{M,\tau}_{mod}(x)=\lim_{\tau\to 0}J^{\tau}_{mod}(F^{-1}(x))
=\begin{cases}
\sum_{k=1}^{M}\frac{\|\left\langle x,e_k\right\rangle_{\mbb C}\|^2}{\alpha_k^2}, \qquad &\text{if}~x\in H_M, \\
+\infty, &\text{otherwise}.
\end{cases}
\end{align*}
Note that for each $x\in H_M$,
\begin{align*}
	\sum_{k=1}^{M}\frac{\|\left\langle x,e_k\right\rangle_{\mbb C}\|^2}{\alpha_k^2}=\frac{1}{\alpha^2}\sum_{k=1}^{M}\frac{\|\left\langle x,e_k\right\rangle_{\mbb C}\|^2}{\eta_k}=\frac{1}{\alpha^2}\left\|Q^{-\frac{1}{2}}P_Mx\right\|_{H^0}^2=\frac{1}{\alpha^2}\left\|Q^{-\frac{1}{2}}x\right\|_{H^0}^2.
\end{align*}
In this way, we have
\begin{align}\label{sec5k18}
\lim_{\tau\to 0}I^{M,\tau}_{mod}(x)
=\begin{cases}
\frac{1}{\alpha^2}\left\|Q^{-\frac{1}{2}}x\right\|_{H^0}^2, \qquad &\text{if}~x\in H_M, \\
+\infty, &\text{otherwise}.
\end{cases}
\end{align} 
	Since  $\eta_{1}\geq\eta_{2}\geq\cdots\geq\eta_{M}>0$, $Q^{\frac{1}{2}}_M(H^0)=H_M$. Hence $I^M$ becomes
	\begin{align*}
	I^M(x)=\begin{cases}
	\frac{1}{\alpha^2}\left\|Q^{-\frac{1}{2}}x\right\|_{H^0}^2, \qquad &\text{if}~x\in H_M, \\
	+\infty, &\text{otherwise}.
	\end{cases}
	\end{align*}
	By the above formula and \eqref{sec5k18},
	$\lim\limits_{\tau\to 0}I^{M,\tau}_{mod}(x)=I^M(x)$. 


(3) \emph{Case 1: There are finitely many $0$ in $\{\eta_k\}_{k\in\mbb N}$.} \\
In this case, for each $M\in\mbb N$, $\eta_{1}\geq\eta_{2}\geq\cdots\geq\eta_{M}>0$. 
Thus, \eqref{sec5k19} and the second case in the proof of Theorem \ref{sec4tho4.4} yield \eqref{wasym1}.

\emph{Case 2: There are infinitely many $0$ in $\{\eta_k\}_{k\in\mbb N}$, i.e., for some $l\in \mbb N$, $\eta_{l}>\eta_{l+1}=\eta_{l+2}=\cdots=0$.}\\
For this case, we take $M=l$ and obtain that $I^M(x)=I(x)$ (see the first case in the proof of Theorem \ref{sec4tho4.4}). Then, it follows from \eqref{sec5k19} that \eqref{wasym1} holds. \hfill$\square$

\subsection{Proof of Theorem \ref{sec5tho5.7}}
In this part, we consider the LDP of $\{B^M_N\}_{N\in\mbb N}$ for full discretizations of \eqref{xde1}, based on the spatial spectral Galerkin method \eqref{Galerkin} and temporal non-symplectic methods.
 Theorem \ref{sec5tho5.7} indicates that  $\left\{u^M_n\right\}_{M,n\in\mbb N}$ can not weakly asymptotically preserve the LDP of $\{B_T\}_{T>0}$.

\textit{Proof of Theorem \ref{sec5tho5.7}}:
Recall  $\hat{\alpha}^k_{n}=\left[\det(A(k^2\tau))\right]^{\frac{n-1}{2}}\sin(n\theta_k)/\sin(\theta_k)$. Under \textbf{Assumption \ref{assum4}},  for sufficiently small $\tau$, $\left|\hat{\alpha}^k_{n}\right|\leq R^{n-1}_{k,\tau}/\sin(\theta_k)$ for some constant $R_{k,\tau}<1$, $k=1,2,\ldots,M$. Denote $T_{M,\tau}=\max_{k=1,2,\ldots,M}R_{k,\tau}$ and then $T_{M,\tau}<1$.
By \eqref{sec5kp} and \eqref{sec5kq}, we have
\begin{align*}
	\left|\mbf Ep^{M,k}_N\right|&=\left|-\det(A)\hat{\alpha}^k_{N-1}p^{M,k}_0+\hat{\alpha}^k_{N}\left(a_{11}p^{M,k}_0+a_{12}q^{M,k}_0\right)\right|\nonumber\\
	&\leq\frac{1}{\sin(\theta_k)}\left|p_0^{M,k}\right|\left(T_{M,\tau}^{N-2}+T_{M,\tau}^{N-1}|a_{11}|\right)+\frac{1}{\sin(\theta_k)}\left|q_0^{M,k}\right||a_{12}|T_{M,\tau}^{N-1}\nonumber\\
	&\leq K(M,\tau).
\end{align*}
Similarly, one has
$\left|\mbf Eq^{M,k}_N\right|
\leq K(M,\tau).$
It follows from \eqref{sec5kp} that
\begin{align*}
\mbf {Var}(p^{M,k}_N)=&\tau\alpha_k^2\sum_{j=0}^{N-1}\left[-\det(A)\hat\alpha^k_{N-2-j}b_1+(a_{11}b_1+a_{12}b_2)\hat\alpha^k_{N-1-j}\right]^2.
\end{align*}
Then, H\"older's inequality and the fact $\left|\hat{\alpha}^k_{n}\right|\leq R^{n-1}_{k,\tau}/\sin(\theta_k)$  yield
\begin{align*}
\left|\mbf {Var}(p^{M,k}_N)\right|\leq& K(M,\tau)\sum_{j=0}^{N-1}\left[\left(\hat\alpha^k_{N-2-j}\right)^2+\left(\hat\alpha^k_{N-1-j}\right)^2\right]\nonumber\\
\leq&K(M,\tau)\sum_{j=0}^{N-1}\left(T_{M,\tau}^{2(N-2-j)}+T_{M,\tau}^{2(N-1-j)}\right)\nonumber\\
=&K(M,\tau)\sum_{j=0}^{N-1}\left(T_{M,\tau}^{2j}+T_{M,\tau}^{2(j-1)}\right)\leq K(M,\tau),
\end{align*}
where we use the fact $\sum_{k=0}^{N-1}r^k<\frac{1}{1-r}$ for each $r\in(0,1)$. Analogously, we obtain
\begin{align*}
	\left|\mbf {Var}(p^{M,k}_N)\right|\leq K(M,\tau),\qquad \left|\mbf {Cor}(p^{M,k}_N,q^{M,k}_N)\right|\leq K(M,\tau).
\end{align*}
Thus, combining the above estimates, we have
\begin{align*}
\left|\mbf{E}\left\langle\lambda_k,p^{M,k}_N+\bm{i}q^{M,k}_N\right\rangle_{\mbb R}\right|+\left|\mbf{Var}\left\langle\lambda_k,p^{M,k}_N+\bm{i}q^{M,k}_N\right\rangle_{\mbb R}\right|<K(M,\tau,\bm{\lambda}),\quad k=1,2,\ldots,M.
\end{align*}
Following the proof of Lemma \ref{sec5lem1}, one can show that the logarithmic moment generating function for $\{A^M_N\}_{N\in\mbb N}$ is $\Lambda^{M,\tau}=0$. Then, we conclude that $\{A^M_N\}_{N\in\mbb N}$ satisfies an LDP on $\mbb C^M$ with the good rate function 
\begin{align}\label{sec5k20}
R(z)=\begin{cases}
0, \qquad &\text{if}~z=0, \\
+\infty, &\text{otherwise}.
\end{cases}
\end{align}
Combining \eqref{sec5k20} and Lemma \ref{contraction}, we have that $\{B^M_N\}_{N\in\mbb N}$ satisfies an LDP on $H^0$ with the good rate function
\begin{align}\label{sec5k21}
I_{ns}^{M,\tau}(x)=R(F^{-1}(x))=\begin{cases}
0, \qquad &\text{if}~x=0, \\
+\infty, &\text{otherwise}.
\end{cases}
\end{align}
It can be verified that \eqref{wasym1} does not hold. \hfill$\square$

\section{Extension to the case of complex-valued noises}\label{Sec6}
In this part, we study the LDP of $\{B_T\}_{T>0}$ for the stochastic \xde equation \eqref{xde1} driven by complex-valued noises. Let $W_1$ be a $U^0$-valued $Q_1$-Winner process and   $W_2$  a $U^0$-valued $Q_2$-Winner process, such that 
$W_1(t)=\sum_{k=1}^{\infty}Q^\frac{1}{2}_1e_k\beta^{(1)}_{k}(t)$ and $W_2(t)=\sum_{k=1}^{\infty}Q^\frac{1}{2}_2e_k\beta^{(2)}_{k}(t)$. Here $Q_1$ and $Q_2$ are two nonnegative symmetric operators on $U^0$ with finite traces. 
$\left\{\beta^{(1)}_k(t)\right\}_{t\geq0}$, $k=1,2,\ldots$ are mutually independent standard Brownian  motions, and $\left\{\beta^{(2)}_k(t)\right\}_{t\geq0}$, $k=1,2,\ldots$ is another family of mutually independent standard Brownian  motions.  
In addition, we assume that  $\left\{\beta^{(1)}_k(t)\right\}_{t\geq0}$ and  $\left\{\beta^{(2)}_j(t)\right\}_{t\geq0}$ mutually independent for all $k,j=1,2,\ldots$ with $k\neq j$. Also assume that for all $k\in\mbb N$, $t>s\geq0$,  $\left(\beta^{(1)}_k(t)-\beta^{(1)}_k(s),\beta^{(2)}_k(t)-\beta^{(2)}_k(s)\right)$
obey the two-dimensional normal distribution with expectation $(0,0)$ and covariance matrix
\begin{align*}
\left(
\begin{array}{cc}
t-s&\rho(t-s)\\
\rho(t-s)&t-s
\end{array}\right),
\end{align*}
for some constant $\rho\in[-1,1]$.
The driving process for stochastic \xde equation \eqref{xde1}  is  $W(t)=W_1(t)+\bm{i}W_2(t)$.

Let $\mcal N^2_{W_1}(0,T;L^2_0)$ denote the set
\begin{align*}
\left\{\left.\Phi:\left[0,T\right]\times \Omega\to\mcal L_2(Q_1^\frac{1}{2}(U^0),U^0)\right|\Phi ~\text{is predicable and} \mbf~ \mbf E\int_{0}^{T}\left\|\Phi(s)\circ Q_1^{\frac{1}{2}}\right\|^2_{\mcal L_2(U^0,U^0)}\ud s<\infty\right\},
\end{align*}
and $\mcal N^2_{W_2}(0,T;L^2_0)$ denote the set
\begin{align*}
\left\{\left.\Phi:\left[0,T\right]\times \Omega\to\mcal L_2(Q_2^\frac{1}{2}(U^0),U^0)\right|\Phi ~\text{is predicable and} \mbf~ \mbf E\int_{0}^{T}\left\|\Phi(s)\circ Q_2^{\frac{1}{2}}\right\|^2_{\mcal L_2(U^0,U^0)}\ud s<\infty\right\}.
\end{align*}
Before giving the LDP of $\{B_T\}_{T>0}$, we first give the following proposition.
\begin{pro}\label{sec6pro1}
	Assume that $\Phi_1\in\mcal N^2_{W_1}(0,T;L^2_0)$, $ \Phi_2\in\mcal N^2_{W_2}(0,T;L^2_0)$. Then the correlation operators
	\begin{align*}
	V(t,s)=\mbf{Cor}(\Phi_1\cdot W_1(t),\Phi_2\cdot W_2(s)),\qquad t,s\in[0,T]
	\end{align*}
	are given by the formula
	\begin{align*}
	V(t,s)=\rho \mbf E\int_{0}^{t\wedge s}\Phi_2(r) Q_2^{\frac{1}{2}}Q_1^{\frac{1}{2}}(\Phi_1(r))^*\ud r.
	\end{align*}
\end{pro}
\begin{proof}
	For simplicity, we take $t=s$. For each $a,b\in U^0$ and $\sigma>r\geq 0$, it follows from the independence of $\left\{\beta^{(1)}_k(t)\right\}_{t\geq0}$ and $\left\{\beta^{(2)}_j(t)\right\}_{t\geq0}$ with $k\neq j$ that  
	\begin{align}\label{sec6k1}
		&\mbf E\left\langle  W_1(\sigma)-W_1(r),a\right\rangle_{U^0}\left\langle  W_2(\sigma)-W_2(r),b\right\rangle_{U^0}\nonumber\\
		=&\mbf E\left(\sum_{k=1}^{\infty}\left(\beta^{(1)}_k(\sigma)-\beta^{(1)}_k(r)\right)\left\langle Q^{\frac{1}{2}}_1e_k,a\right\rangle_{U^0}\right)\left(\sum_{j=1}^{\infty}\left(\beta^{(2)}_j(\sigma)-\beta^{(2)}_j(r)\right)\left\langle Q^{\frac{1}{2}}_2e_j,b\right\rangle_{U^0}\right)\nonumber\\
		=&\sum_{k=1}^{\infty}\mbf E\left(\beta^{(1)}_k(\sigma)-\beta^{(1)}_k(r)\right)\left(\beta^{(2)}_k(\sigma)-\beta^{(2)}_k(r)\right)\left\langle e_k,Q^{\frac{1}{2}}_1a\right\rangle_{U^0}\left\langle e_k,Q^{\frac{1}{2}}_2b\right\rangle_{U^0}\nonumber\\
		=&\rho(\sigma-r)\left\langle Q^{\frac{1}{2}}_1a,Q^{\frac{1}{2}}_2b\right\rangle_{U^0}=\rho(\sigma-r)\left\langle Q^{\frac{1}{2}}_2Q^{\frac{1}{2}}_1a,b\right\rangle_{U^0}.
	\end{align}
	
	We first prove that  the conclusion hold in the case that both $\Phi_1$ and $\Phi_2$ are elementary processes. For this end, assume that there is a partition $0=t_0<t_1<\cdots<t_N=t$, $N\in\mbb N$, such that
	\begin{align*}
		\Phi_1(r)=\sum_{n=0}^{N-1}\Phi_1^{n}\mbf 1_{(t_n,t_{n+1}]}(r),\qquad\Phi_2(r)=\sum_{n=0}^{N-1}\Phi_2^{n}\mbf 1_{(t_n,t_{n+1}]}(r),
	\end{align*}
	where $\Phi_i^n:\Omega\to\mcal L(U^0,U^0)$ is $\mcal F_{t_n}$-measurable, and $\Phi_i^n$ takes only  a finite number of values in $L(U^0,U^0)$, $i=1,2$, $0\leq n\leq N-1$. Then we have that for each $a,b\in U^0$,
	\begin{align}\label{sec6k2}
		&\mbf E\left\langle \int_{0}^{t}\Phi_1(r)d W_1(r),a\right\rangle_{U^0}\left\langle \int_{0}^{t}\Phi_2(r)d W_2(r),b\right\rangle_{U^0}\nonumber\\
		=&\mbf E\left(\sum_{j=0}^{N-1}\left\langle \Phi_1^j(W_1(t_{j+1})-W_1(t_j)),a\right\rangle_{U^0}\right)\left(\sum_{k=0}^{N-1}\left\langle \Phi_2^k(W_2(t_{k+1})-W_2(t_k)),b\right\rangle_{U^0}\right)\nonumber\\
		=&\sum_{j=0}^{N-1}\sum_{k=0}^{N-1}\mbf E\left\langle W_1(t_{j+1})-W_1(t_j),(\Phi_1^j)^*a\right\rangle_{U^0}\left\langle W_2(t_{k+1})-W_2(t_k),(\Phi_2^j)^*b\right\rangle_{U^0}\nonumber\nonumber\\
		=:& \sum_{j,k=0}^{N-1}\mbf ES_{j,k}.
	\end{align}
	If $k\neq j$, we claim $\mbf ES_{j,k}=0$. For this end, we may assume that $k>j$ without loss of generality. Then $\left\langle W_1(t_{j+1})-W_1(t_j),(\Phi_1^j)^*a\right\rangle_{U^0}$ and $\Phi_2^k$ are $\mcal F_{t_k}$-measurable. In addition $\left(W_2(t_{k+1})-W_2(t_k)\right)$ is $\mcal F_{t_k}$-independent. It follows from the properties of conditional expectation that
	\begin{align*}
		&\mbf E(S_{j,k}|\mcal  F_{t_k})
		\nonumber\\
		=&\left\langle W_1(t_{j+1})-W_1(t_j),(\Phi_1^j)^*a\right\rangle_{U^0}\mbf E\left[\left.\left\langle W_2(t_{k+1})-W_2(t_k),(\Phi_2^j)^*b\right\rangle_{U^0}\right| \mcal F_{t_k}\right]\nonumber\\
		=&\left\langle W_1(t_{j+1})-W_1(t_j),(\Phi_1^j)^*a\right\rangle_{U^0} \Big(\mbf E\left.\left\langle W_2(t_{k+1})-W_2(t_k),u\right\rangle_{U^0} \Big)\right|_{u=(\Phi_2^j)^*b}\nonumber\\
		=&0,
	\end{align*}
which leads to
	\begin{align}\label{sec6k3}
	\mbf ES_{j,k}=\mbf E\left(\mbf E(S_{j,k}|\mcal  F_{t_k})\right)=0,\qquad k\neq j.
	\end{align}
	Similarly, using \eqref{sec6k1} we obtain
	\begin{align*}
	 &\mbf E(S_{k,k}|\mcal   F_{t_k})\nonumber\\
	=&\left.\Big(\mbf E\left\langle W_1(t_{k+1})-W_1(t_k),u\right\rangle_{U^0} \left\langle W_2(t_{k+1})-W_2(t_k),v\right\rangle_{U^0}\Big)\right|_{u=(\Phi_1^k)^*a,\,v=(\Phi_2^j)^*b}\nonumber\\
	=&\rho(t_{k+1}-t_k)\left.
	\left\langle Q_2^{\frac{1}{2}}Q_1^{\frac{1}{2}}u,v\right\rangle_{U^0}\right|_{u=(\Phi_1^k)^*a,\,v=(\Phi_2^k)^*b}\nonumber\\
	=&\rho(t_{k+1}-t_k)\left\langle \Phi_2^k Q_2^{\frac{1}{2}}Q_1^{\frac{1}{2}}(\Phi_1^k)^*a,b\right\rangle_{U^0}.
	\end{align*}
	Hence, it holds that
	\begin{align}\label{sec6k4}
		\mbf ES_{k,k}=\rho(t_{k+1}-t_k)\mbf E\left\langle \Phi_2^k Q_2^{\frac{1}{2}}Q_1^{\frac{1}{2}}(\Phi_1^k)^*a,b\right\rangle_{U^0}.
	\end{align}
	Substituting \eqref{sec6k3} and \eqref{sec6k4} into \eqref{sec6k2} yields
	\begin{align*}
		&\mbf E\left\langle \int_{0}^{t}\Phi_1(r)d W_1(r),a\right\rangle_{U^0}\left\langle \int_{0}^{t}\Phi_2(r)d W_2(r),b\right\rangle_{U^0}\nonumber\\
		=&\rho\sum_{k=0}^{N-1}(t_{k+1}-t_k)\mbf E\left\langle \Phi_2^k Q_2^{\frac{1}{2}}Q_1^{\frac{1}{2}}(\Phi_1^k)^*a,b\right\rangle_{U^0} \nonumber\\
		=&\rho\mbf E\left\langle\int_{0}^{t}\Phi_2(r) Q_2^{\frac{1}{2}}Q_1^{\frac{1}{2}}(\Phi_1(r))^*a,b\right\rangle_{U^0},
	\end{align*}
	which proves the conclusion when $\Phi_i$, $i=1,2$,  are elementary processes.
	
	If $\Phi_i$, $i=1,2$, are general processes, one can take elementary process $\Phi_i^{(n)}$ such that
	\begin{align*}
	\lim_{n\to \infty}\mbf E\int_{0}^{T}\left\|\left(\Phi_i^{(n)}(s)-\Phi_i(s)\right)\circ Q_i^{\frac{1}{2}}\right\|^2_{\mcal L_2(U^0,U^0)}\ud s=0,\qquad i=1,2.
	\end{align*} 
	Then by a standard argument of approximation, one can prove that the conclusion holds for any $\Phi_1\in\mcal N^2_{W_1}(0,T;L^2_0)$, $ \Phi_2\in\mcal N^2_{W_2}(0,T;L^2_0)$ (see also the proof of \cite[Proposition 4.28]{Prato1}).
\end{proof}

Similar to the case of real-valued noises, we assume that $Q_i^{\frac{1}{2}}\in\mcal L_2(U^0,U^1)$, $i=1,2$. Then, we have the following results.
\begin{theo}\label{sec6tho6.2}
	Under the above conditions,  $\{B_T\}_{T>0}$ satisfies an LDP on $H^0$ with the good rate function
	\begin{align*}
	I(x)=\begin{cases}
	\frac{1}{\alpha^2}\left\|\widetilde{Q}^{-\frac{1}{2}}x\right\|_{H^0}^2, \qquad &\text{if}~x\in \widetilde{Q}^{\frac{1}{2}}(H^0), \\
	+\infty, &\text{otherwise},
	\end{cases}
	\end{align*}
	where $\widetilde{Q}=Q_1+Q_2$.
\end{theo}
\begin{proof}
	This proof is analogous to that of Theorem \ref{LDP for BT}. Hence we only give the sketch of proof.
	The main difference lies in the computation of the variance $\mbf{Var}\left\langle u(t),h\right\rangle_\mathbb{R}$. In fact, it holds that
	\begin{align*}
		u(t)=&S(t)u_0-\alpha\int_{0}^{t}\sin((t-s)\Delta)\ud W_1(s)-\alpha\int_{0}^{t}\cos((t-s)\Delta)\ud W_2(s)\nonumber\\
		&+\bm{i}\alpha\int_{0}^{t}\cos((t-s)\Delta)\ud W_1(s)-\bm{i}\alpha\int_{0}^{t}\sin((t-s)\Delta)\ud W_2(s).
	\end{align*}
	Hence, for each $h\in H^0$,
	\begin{align*}
	\left\langle u(t),h\right\rangle_\mathbb{R}=&\left\langle S(t)u_0,h\right\rangle_\mathbb{R}-\alpha	\left\langle \int_{0}^{t}\sin((t-s)\Delta)\ud W_1(s),\Re h\right\rangle_\mathbb{R}-\alpha\left\langle \int_{0}^{t}\cos((t-s)\Delta)\ud W_2(s),\Re h\right\rangle_\mathbb{R}\nonumber\\
	&+\alpha\left\langle \int_{0}^{t}\cos((t-s)\Delta)\ud W_1(s),\Im h\right\rangle_\mathbb{R}-\alpha\left\langle \int_{0}^{t}\sin((t-s)\Delta)\ud W_2(s),\Im h\right\rangle_\mathbb{R}.
	\end{align*}
	Using Proposition \ref{sec6pro1}, one has
	\begin{align*}
		&\mbf{Var}\left\langle u(t),h\right\rangle_\mathbb{R}\nonumber\\
		=&\alpha^2\left\langle \int_{0}^{t}\sin^2((t-s)\Delta)Q_1\ud s\Re h,\Re h\right\rangle_\mathbb{R}+\alpha^2\left\langle \int_{0}^{t}\cos^2((t-s)\Delta)Q_2\ud s\Re h,\Re h\right\rangle_\mathbb{R}\nonumber\\
		&+\alpha^2\left\langle \int_{0}^{t}\cos^2((t-s)\Delta)Q_1\ud s\Im h,\Im h\right\rangle_\mathbb{R}+\alpha^2\left\langle \int_{0}^{t}\sin^2((t-s)\Delta)Q_2\ud s\Im h,\Im h\right\rangle_\mathbb{R}\nonumber\\
		&+2\alpha^2\rho\left\langle \int_{0}^{t}\sin((t-s)\Delta)\cos((t-s)\Delta)Q_2^{\frac{1}{2}}Q_1^{\frac{1}{2}}\ud s\Re h,\Re h\right\rangle_\mathbb{R}\nonumber\\
		&-2\alpha^2\left\langle \int_{0}^{t}\sin((t-s)\Delta)\cos((t-s)\Delta)Q_1\ud s\Re h,\Im h\right\rangle_\mathbb{R}\nonumber\\
		&+2\alpha^2\rho\left\langle \int_{0}^{t}\sin^2((t-s)\Delta) Q_2^{\frac{1}{2}}Q_1^{\frac{1}{2}}\ud s\Re h,\Im h\right\rangle_\mathbb{R}-2\alpha^2\rho\left\langle \int_{0}^{t}\cos^2((t-s)\Delta) Q_2^{\frac{1}{2}}Q_1^{\frac{1}{2}}\ud s\Im h,\Re h\right\rangle_\mathbb{R}\nonumber\\
		&+2\alpha^2\left\langle \int_{0}^{t}\sin((t-s)\Delta)\cos((t-s)\Delta)Q_2\ud s\Re h,\Im h\right\rangle_\mathbb{R}\nonumber\\
		&-2\alpha^2\rho\left\langle \int_{0}^{t}\sin((t-s)\Delta)\cos((t-s)\Delta)Q_2^{\frac{1}{2}}Q_1^{\frac{1}{2}}\ud s\Im h,\Im h\right\rangle_\mathbb{R}\nonumber\\
		=&\frac{t\alpha^2}{2}\left(\left\langle \widetilde{Q}\Re h,\Re h\right\rangle_\mathbb{R}+\left\langle \widetilde{Q}\Im h,\Im h\right\rangle_\mathbb{R}\right)+\widetilde{R},\nonumber\\
		=&\frac{t\alpha^2}{2}\left\|\widetilde{Q}^{\frac{1}{2}}\lambda\right\|_{H^0}^2+\widetilde{R}, 
	\end{align*}
	where $|\widetilde R|\leq K(Q_1,Q_2,\Delta)$ with $K(Q_1,Q_2,\Delta)$ independent of $t$.  Similar to the proof of Theorem \ref{LDP for BT}, we finish the proof by means of the abstract G\"artner--Ellis theorem.
\end{proof}
\begin{rem}
	In Theorem \ref{sec6tho6.2}, we give the LDP of $\{B_T\}_{T>0}$. Similarly, the LDP for $\{B^M_N\}_{M,N\in\mbb N}$ of numerical method can also be obtained in the case of complex-valued noises. 
\end{rem}

\section{Future work}\label{Sec7}

The calculation of large deviations rate functions is an interesting and important problem.  One of the  common techniques of approximating the large deviations rate functions is by the Legendre transform of the approximated logarithmic moment generating functions which may be obtained by, e.g., Monte--Carlo methods  provided the prior distributions  of observables are known (\cite{ratefun}). 
For a stochastic system, the
prior distributions of the considered observables are generally unknown, the approximated logarithmic moment generating functions can be obtained by the combination of numerical discretizations and Monte--Carlo methods. Do all of numerical discretizations work? Theorem \ref{sec5tho5.5} of this paper shows that 
the full discretizations  $\{u^M_n\}_{M,n\in\mbb N}$, based on the temporal symplectic discretizations and the spatial spectral Galerkin approximation, can weakly asymptotically preserve the LDP of $\{B_T\}_{T>0}$.  
This result indicates that for an observable associated with a stochastic Hamiltonian partial differential equation, the symplectic discretization is a prior choice. What is the convergence between the rate functions and their numerical approximations? How to combine other techniques, e.g.,  the adaptive sampling algorithm (see \cite{LDPapp18}) and multi-level Monte--Carlo methods, to  improve the computational efficiency?

\bibliographystyle{plain}
\bibliography{mybibfile}

\end{document}